\documentclass[12pt]{amsart}
\usepackage{geometry}   
\usepackage[colorlinks,citecolor = red, linkcolor=blue,hyperindex]{hyperref}
\usepackage{euscript,eufrak,verbatim, mathrsfs, amscd}
\usepackage[psamsfonts]{amssymb}
\usepackage{bbm}
\usepackage{graphicx}
 \usepackage{float}
\usepackage{float, tikz}

 \usepackage[all, cmtip]{xy}

\usepackage{upref, xcolor, dsfont}
\usepackage{amsfonts,amsmath,amstext,amsbsy, amsopn,amsthm}
\usepackage{enumerate}

\usepackage{url}

\usepackage{mathtools}
\usepackage{bookmark}

 \usepackage{euscript}
\usepackage{helvet}         
\usepackage{courier}        
\usepackage{type1cm}        
\usepackage{multicol}        
\usepackage[bottom]{footmisc}

\newtheorem{theorem}{Theorem}[section]
\newtheorem*{theorem*}{Theorem B}
\newtheorem{lemma}[theorem]{Lemma}

\newtheorem{proposition}[theorem]{Proposition}
\newtheorem{corollary}[theorem]{Corollary}
\newtheorem*{definition*}{Definition}
\newtheorem*{remark*}{Remark}

\newtheorem*{observation*}{Observation}

\newtheorem*{assumption*}{Assumption}
\newtheorem{example}[theorem]{Example}

\theoremstyle{definition}
\newtheorem{definition}{Definition}[section]
\newtheorem{question}{Question}

\theoremstyle{remark}
\newtheorem{remark}{Remark}[section]

\newcommand{\R}{\mathbb{R}}
\newcommand{\N}{\mathbb{N}}
\newcommand{\Z}{\mathbb{Z}}

\newcommand{\C}{\mathbb{C}}

\newcommand{\T}{\mathbb{T}}

\newcommand{\B}{\mathbb{B}}
\newcommand{\BS}{\mathbb{S}}
\newcommand{\F}{\mathbb{F}}

\newcommand{\TT}{\mathit{Toep}}

\newcommand{\spec}{\mathrm{spec}}

\newcommand{\rad}{\mathrm{rad}}

\newcommand{\pless}{\preccurlyeq}

\newcommand{\an}{\text{\, and \,}}

\begin{document}

\title[Branching-Toeplitz operators]{Hyper-positive definite functions II: A complete study of branching-Toeplitz operators}

\author
{Yanqi Qiu}
\address
{Yanqi QIU: Institute of Mathematics and Hua Loo-Keng Key Laboratory of Mathematics, AMSS, Chinese Academy of Sciences, Beijing 100190, China.}
\email{yanqi.qiu@amss.ac.cn; yanqi.qiu@hotmail.com}

\author{Zipeng Wang}
\address{Zipeng Wang: Department of Mathematics, Soochow University, Suzhou 215006, China}
\email{zipengwang2012@gmail.com}

 \begin{abstract}
We introduce and give a more or less complete study of a family of branching-Toeplitz operators on the Hilbert space $\ell^2(T_q)$  indexed by a rooted homogeneous tree $T_q$ of degree $q\ge 2$.  The finite dimensional analogues of such operators form a very natural family of structured sparse matrices called branching-Toeplitz matrices and will also be investigated. The branching-Toeplitz operators/matrices in this paper should be viewed as natural generalizations of the standard Toeplitz operators/matrices. We will apply our results to construct a family of determinantal point processes on homogeneous trees which are branching-type strong stationary stochastic processes.
\end{abstract}


\maketitle

\section{Introduction}
\subsection{Sparse matrices associated with a partial order}\label{sec-sparse}

In what follows, by a tree, we will always mean the set of the vertices of the tree.  Fix an integer $q\ge 2$ and let $T_q$ be the {\it infinite rooted $q$-homogeneous tree}: that is, there is a distinguished vertex $o \in T_q$ called the root vertex which is the unique vertex without ancestor and each vertex in $T_q$ has exactly $q$ children.  Let  $d(\cdot, \cdot)$ denote  the usual graph distance  on $T_q$ and
for each $\sigma \in T_q$, we call $d(o, \sigma)$ the {\it generation number} of $\sigma$ and will be denoted simply by
\[
| \sigma| := d(o, \sigma), \quad \sigma \in T_q.
\]
Let $\pless$ denote the natural partial order on $T_q$. That is,  two vertices   $\sigma_1, \sigma_2\in T_q$ satisfy $\sigma_1 \pless \sigma_2$ if and only if they are contained in the same {\it rooted geodesic ray} (that is, a geodesic ray starting from the root vertex) and $|\sigma_1| \le |\sigma_2|$.   Let $\mathcal{C}$ be the set corresponding to the relation  of comparability for the partial order $\pless$:
\[
\mathcal{C}  = \Big\{(\sigma_1, \sigma_2) \in T_q \times T_q\Big| \text{$\sigma_1 \pless \sigma_2$  or $\sigma_2 \pless \sigma_1$}\Big\}.
\]
 Note that $\mathcal{C}$ is symmetric in the sense that $(\sigma_1, \sigma_2) \in \mathcal{C}$ if and only if $(\sigma_2, \sigma_1)\in \mathcal{C}$.

The partial order $\pless$ on $T_q$ gives rise to a natural family of  sparse matrices as follows.   For  any integer $n \ge 1$, let $\B_n(T_q) \subset T_q$ be the finite subset defined by
\begin{align}\label{B-n-S-n}
\B_n(T_q): =\Big \{\sigma \in T_q \Big| | \sigma| \le n \Big\}.
\end{align}
By convention, we also denote $T_q$ by $B_\infty(T_q)$.  Consider matrices  of the form:
\[
A = \big[a_{\sigma_1, \sigma_2}\big]_{\sigma_1, \sigma_2 \in \B_n(T_q)} \text{\, with $a_{\sigma_1, \sigma_2} = 0$ whenever $(\sigma_1, \sigma_2) \notin \mathcal{C}$}.
\]
We call such matrices {\it branching-type matrices} (the word ``branching-type'' comes from its relation with {\it branching-type stationary stochastic processes} on $T_q$ introduced in \cite{QW-I}).   Clearly, branching-type matrices are sparse (for $q\ge 2$) in the sense that the number of non-zero entries of any such matrix is $o([\# \B_n(T_q)]^2)$ as $n$ goes to infinity. The  branching-type matrices arise naturally in the study of left creation operators on the full Fock space of $\C^q$, see \S \ref{sec-fock} for details.

\subsection{Branching-Toeplitz kernels}
Given any function $\alpha: \Z \rightarrow \C$, we may introduce a {\it branching-Toeplitz kernel} $K_\alpha: T_q\times T_q\rightarrow \C$ by
\[
K_\alpha(\sigma_1, \sigma_2) =  \alpha( | \sigma_1| - |\sigma_2|) \cdot \mathds{1}_{\mathcal{C}}(\sigma_1, \sigma_2), \quad  \sigma_1, \sigma_2 \in T_q.
\]
Let $\TT(T_q)$ denote the set of all branching-Toeplitz kernels on $T_q$:
\[
\TT(T_q): = \Big\{ K_\alpha \Big|
\text{$\alpha$ ranges over all  $\C$-valued functions on $\Z$} \Big\}.
\]
 If $q =1$, then $T_1 = \N = \{0, 1, 2, \cdots \}$ and we go back to the standard Toeplitz matrices. If $q \ge 2$, then branching-Toeplitz kernels are  special infinite branching-type matrices introduced in \S\ref{sec-sparse}. Note that restricted on any rooted geodesic ray of $T_q$, a branching-Toeplitz kernel becomes a standard Toeplitz kernel. Note that the branching-Toeplitz kernels in this paper are different from the radial Toeplitz kernels on homogeneous tree (non-rooted) studied in \cite{Q-radial-Toep}.

Ut is convenient for us to introduce the following notation. Given any formal  Fourier series on  $\T = \R/2 \pi \Z$:
\begin{align}\label{fourier-series}
f \sim \sum_{n \in \mathbb{Z}}\widehat{f}(n) e^{in\theta}, \quad \text{with $\widehat{f}(n)\in \C$ for all $n \in \Z$},
\end{align}
 we define a branching-Toeplitz kernel $\Gamma_q[f]: T_q \times T_q \rightarrow \C$ by
\begin{align}\label{def-T-f}
\Gamma_q[f](\sigma_1, \sigma_2): =  \sqrt{q^{- d(\sigma_1, \sigma_2)}} \cdot \widehat{f}(| \sigma_1| - | \sigma_2| ) \cdot  \mathds{1}_{\mathcal{C}}(\sigma_1, \sigma_2), \quad \sigma_1, \sigma_2 \in T_q.
\end{align}
Note that $d(\sigma_1, \sigma_2) = | |\sigma_1| - |\sigma_2||$ for all $(\sigma_1, \sigma_2) \in \mathcal{C}$.
Clearly, any branching-Toeplitz kernel on $T_q$ is of the form $\Gamma_q[f]$ for a unique formal Fourier series $f$. Thus we have
\[
\TT(T_q) = \Big\{\Gamma_q[f]\Big| \text{$f$ is a formal Fourier series on $\T$}\Big\}.
\]
In what follows, we shall call the formal Fourier series $f$ the symbol of $\Gamma_q[f]$.

\begin{remark}
By formal Fourier series, we mean that no a priori assumption is assumed on the sequence $(\widehat{f}(n))_{n\in \Z}$. Thus the notation  $\widehat{f}$ does not mean that it is the Fourier transform of a function $f$ on $\T$. However, if the formal Fourier series \eqref{fourier-series} coincides with the Fourier series of a function in $L^1(\T)$ or a Radon measure on $\T$, then we will identify it with the function  or the Radon measure respectively.
\end{remark}

Throughout the paper, we do not distinguish the meaning of positive (resp. positive definite) and non-negative (resp. non-negative definite). In \cite{QW-I}, we obtain the following criterion for positive definite branching-Toeplitz kernels on $T_q$. 

\begin{theorem}[{\cite[Thm 1.8]{QW-I}}]\label{thm-QW-I}
Let $q \ge 2$ be an integer.  The kernel $\Gamma_q[f]$  is positive definite if and only if there exists a positive Radon measure $\mu$ on  $\T = \R/2 \pi \Z$ such that
\[
\widehat{f}(n) =\int_\T  e^{-i  n \theta} d \mu(\theta), \quad \forall n \in \Z.
\]
In the above situation, by identifying the formal Fourier series $f$ with the positive Radon measure $\mu$, we will also denote $\Gamma_q[f]$ by $\Gamma_q[\mu]$.
\end{theorem}
By using the method of left creation operators on the full Fock space of $\C^q$, we will obtain a useful relation between $\Gamma_q[f]$ and the standard Toeplitz kernels on $\N$. This relation in particular allows us to give a new proof of Theorem \ref{thm-QW-I}. See \S \ref{sec-new-pf}.

\subsection{Branching-Toeplitz operators}

As usual, denote by $\ell^2(T_q)$  the Hilbert space:
\[
\ell^2(T_q): = \Big\{(v_\sigma)_{\sigma \in T_q} \in \C^{T_q}\Big| \sum_{\sigma \in T_q}|v_\sigma|^2< \infty\Big\}
\]
and let $B(\ell^2(T_q))$ be the set of bounded operators on $\ell^2(T_q)$.  We shall always identify a bounded operator $A \in B(\ell^2(T_q))$ with its kernel  defined by
$A(\sigma_1, \sigma_2)  = \langle A \delta_{\sigma_2}, \delta_{\sigma_1}\rangle$ for all $\sigma_1, \sigma_2 \in T_q$. Using this identification, we may define
\[
B_\TT(\ell^2(T_q)) : = \TT(T_q) \cap B(\ell^2(T_q)).
\]
 The first goal of this paper is to study the problem when the kernel $\Gamma_q[f]$ represents a bounded operator on $\ell^2(T_q)$.

Before stating the results in the case  $q\ge 2$,  let us briefly recall the classical results in the case $q =1$. If $q=1$, then  $T_1 = \N$ and $\Gamma_1[f]$ is  the standard Toeplitz kernel:
\begin{align}\label{cla-toep}
\Gamma_1[f]= T(f): = \Big[\widehat{f}(k-l)\Big]_{k, l \in \N}.
\end{align}
It is  a standard result that  $\Gamma_1[f]$ is a bounded operator on $\ell^2(\N)$ if and only if  $f\in L^\infty(\T)$. Moreover, in this case,  we have the powerful tool of function spaces to deal with the operator $\Gamma_1[f]$. Namely, let $H^2(\T)$ be the Hardy space:
\[
H^2(\T):   = \Big\{f = \sum_{n = 0}^\infty \widehat{f}(n) e^{i n \theta}\Big| \sum_{n  =0}^\infty | \widehat{f}(n) |^2 <\infty \Big\} \subset L^2(\T).
\]
For any $f\in L^\infty(\T)$, the operator $\Gamma_1[f]$  is unitarily equivalent to the standard Toeplitz operator on the Hardy space $H^2(\T)$ defined by  the composition
\begin{align}\label{Toep-hardy}
T_f: H^2(\T) \xrightarrow{M_f} L^2(\T) \xrightarrow{\,P_{+}\,} H^2(\T),
\end{align}
 where $M_f$ is the operator of pointwise multiplication by $f$ and $P_{+}$ is the orthogonal projection onto $H^2(\T)$. In other words,  we have the following commutative diagram:
\begin{align}\label{2-way-Toep}
\begin{CD}
H^2(\T) @>\,\, T_f\,\,>>  H^2(\T)
\\
@V{\text{Fourier transform}}V{\simeq}V  @V{\simeq}V{\text{Fourier transform}}V
\\
\ell^2(\N) @> T(f)= \Gamma_1[f]>>  \ell^2(\N)
\end{CD}.
\end{align}

Now we turn back to the case $q \ge 2$. Since there is no counterpart of Hardy space any more,  the study of the boundedness of $\Gamma_q[f]$ requires more efforts.   We have the following

\begin{theorem}\label{thm-all-symbol}
For any integer $q \ge 2$, we have the set-theoretical equality:
\begin{align}\label{set-th-eq}
B_\TT(\ell^2(T_q)) = \Big\{ \Gamma_q[f]\Big| f \in L^\infty(\T) \Big\}.
\end{align}
Moreover, for any $f\in L^\infty(\T)$, we have
\begin{align}\label{norm-2-side}
 \| \Gamma_q[f]\|   =  \| f \|_\infty.
\end{align}
\end{theorem}

The following multiplicativity result should be compared with the classical result of Brown and Halmos  \cite[Thm 8]{BH-1963}.  Recall that the  Hardy space $H^\infty(\T)$ is defined by
\[
H^\infty(\T) : = \Big\{f \in L^\infty(\T)\Big| f \sim  \sum_{n=0}^\infty \widehat{f}(n) e^{i n \theta}\Big\}.
\]

\begin{theorem}\label{thm-gen-mul}
Let $f, g \in L^\infty(\T)$. Then the equality
$\Gamma_q[\bar{f}g] = \Gamma_q[\bar{f}]\Gamma_q[g]$
holds if and only if either $f \in H^\infty(\T)$ or $g \in H^\infty(\T)$.
\end{theorem}

By using the method of full Fock space, we are able to  obtain more precise information for $\Gamma_q[f]$.   Namely, we have

\begin{theorem}\label{thm-direct-sum}
Let $q\ge 2$.  Then the operators $\Gamma_q[f]$ for all $f\in L^\infty(\T)$ are simultaneously unitarily equivalent to the direct sum of countably infinitely many standard Toeplitz operatars $\Gamma_1[f] = T(f)$. That is, there exists a unitary operator $U: \ell^2(T_q) \rightarrow  \bigoplus_{k=1}^\infty \ell^2(\N) $ such that
\[
\Gamma_q[f]  = U^{-1}
\left[
\begin{array}{cccc}
\Gamma_1[f] & & &
\\
& \Gamma_1[f]& &
\\
& & \Gamma_1[f] &
\\
& & &  \ddots
\end{array}
 \right] U
 =
U^{-1}\left[
\begin{array}{cccc}
T(f) & & &
\\
& T(f) & &
\\
& & T(f) &
\\
& & &  \ddots
\end{array}
 \right] U.
\]
\end{theorem}

\begin{remark}
Clearly, Theorem \ref{thm-all-symbol} and Theorem \ref{thm-gen-mul}  can be derived from Theorem  \ref{thm-direct-sum} and the classical results of Brown and Halmos  \cite{BH-1963}. However,  the proofs (without using Theorem  \ref{thm-direct-sum}) of Theorem \ref{thm-all-symbol} and Theorem \ref{thm-gen-mul} seem to be of independent interests and are included in this paper.
\end{remark}

\begin{remark}
The unitary operator $U$ in the statement of Theorem \ref{thm-direct-sum} can be written explicitly. See the proof of Theorem \ref{thm-direct-sum}.
\end{remark}

\subsection{Corollaries to Theorem \ref{thm-direct-sum}}

Theorem \ref{thm-direct-sum} has the following immediate corollaries, the routine proofs of which will all be omitted.

\begin{corollary}\label{cor-isometry}
$\Gamma_q[f]$ is an isometry on $\ell^2(T_q)$ if and only if $f$ is an inner function.
\end{corollary}

\begin{corollary}
Let $p, q \ge 2$ be any two integers and $f\in L^\infty(\T)$.  Then the operators $\Gamma_q[f]$ and $\Gamma_p[f]$ on $\ell^2(T_q)$ and $\ell^2(T_p)$ respectively are unitarily equivalent.
\end{corollary}

\begin{corollary}
For any $f\in L^\infty(\T)$, we have the following coincidence of spectra:
\[
\spec(\Gamma_q[f]) =  \spec(\Gamma_1[f])=\spec(T(f)).
\]
\end{corollary}

\begin{corollary}
Let $q \ge 2$ and $f\in L^\infty(\T)$. Then  the operator $\Gamma_q[f]$ is Fredholm if and only if it  is  invertible. Moreover, $\Gamma_q[f]$ is invertible if and only if so is $T(f)$.
\end{corollary}

\begin{corollary}\label{cor-semi-comm}
Let $q\ge 2$ and $f, g\in L^\infty(\T)$. Then $\Gamma_q[\bar{f}] \Gamma_q[g] - \Gamma_q[\bar{f}g]$ is compact if and only if  $\Gamma_q[\bar{f}] \Gamma_q[g]  =  \Gamma_q[\bar{f}g]$ and hence if and only if either $f \in H^\infty(\T)$ or $g \in H^\infty(\T)$.
\end{corollary}

\begin{corollary}\label{cor-inv-sub}
The family of invariant subspaces for the operator $\Gamma_q[e^{i\theta}]$ can be naturally parametrized by the set of sequences $(\Theta_n)_{n=1}^\infty$ of inner functions on $\T$.
\end{corollary}
\begin{remark}The principal criterion for the invertibility of $T(f)$ was given by Widom \cite{Widom-1960} and Devinatz \cite{Dev-1964}. For more equivalent conditions, we refer to Nikolski  \cite[Thm 3.3.6]{Nikolski-20}. We also note that detailed descriptions of the spectrum of $T(f)$ can be found in Douglas \cite[Chapter 7]{Douglas}.
\end{remark}

\begin{remark}
Our Corollary \ref{cor-semi-comm} stands in sharp contrast with the classical result on the Hardy space $H^2(\T)$. Axler, Chang and Sarason \cite{ACS-1978} (for the sufficient part) and Volberg \cite{Vol-1982} (for the necessary part) show that the semi-commutator $T_f T_g-T_{fg}$ is compact on $H^2(\T)$ if and only if $H^\infty[\bar{f}]\cap H^{\infty}[g]\in H^\infty(\T)+C(\T)$, where $H^{\infty}[g]$ is the closed subalgebra in $L^\infty(\T)$ generated by $H^\infty(\T)$ and $g$.
\end{remark}

\begin{remark}
Using the classical result of Beurling, we can make Corollary \ref{cor-inv-sub} more precise by giving  an explicit description of all invariant subspaces of $\Gamma_q[e^{i\theta}]$. Such kind of result has a significant overlap with \cite[Thm 2.3]{Pop-1989} and  \cite[Thm 2.1]{DP-1999}. In fact,  using the relation between invariant subspaces and wandering subspaces, Popescu \cite{Pop-1989} firstly obtains a Beurling-Lax-Halmos type characterization of the invariant subspaces under the left creation operators on the full Fock space. However, our point of view is somewhat different and our proof seems to be more elementary and simpler.  For further developments along this line, one can consult
\cite{Pop-2010}, \cite{Pop-2017}, \cite{Pop-2018} and \cite{Pop-2019}.
\end{remark}

\subsection{Generalizations of Theorem \ref{thm-direct-sum}}
 Let us identify $T_q$ with the free unital semi-group $\F_q^{+}$  generated $q$ free elements $s_1, \cdots, s_q$ (the neutral element is denoted by $e\in \F_q^{+}$) by fixing any  bijection
\begin{align}\label{def-iota}
\iota: \F_q^+ \rightarrow T_q
\end{align}
 with the following properties:
\begin{itemize}
\item  $\iota (e) = o$ the root vertex of $T_q$;
\item  the children of $\iota(w)$ are exactly $\iota(s_1 w), \cdots, \iota (s_q w)$ for any word $w\in \F_q^+$.
\end{itemize}

Let $\BS(\C^q)$ denote the unit sphere in $\C^q$. That is,
\[
\BS(\C^q): = \Big\{ a = (a_1, \cdots, a_q)\in \C^q \Big| \sum_{k=1}^q |a_k|^2 = 1\Big\}.
\]
Given $a\in \BS(\C^q)$ and a non-empty word $w = s_{i_1} s_{i_2} \cdots s_{i_n} \in \F_q^{+}$, we define
\[
[w](a): = a_{i_1} a_{i_2} \cdots a_{i_n} \in \C
\]
and for the empty word $e\in \F_q^{+}$, we set $[e](a)=1$.
For any formal Fourier series $f$ given in \eqref{fourier-series}, define a kernel $\Gamma_{a}[f]: \F_q^{+} \times \F_q^{+}\rightarrow \C$ by
\begin{align}\label{def-gamma-a-f}
\Gamma_{a}[f] (u, v): = \left\{
\begin{array}{cc}
 \widehat{f}(|u| - |v|) \cdot [w](a), &   \text{if $u = w v$}
\vspace{2mm}
\\
 \widehat{f}(|u| - |v|)  \cdot \overline{[w](a)}, &    \text{if $v = w u$}
\vspace{2mm}
\\
0, & \text{otherwise}
\end{array}
\right..
\end{align}
In particular, in the above notation and using the identification \eqref{def-iota},  we have
\[
\Gamma_q[f] = \Gamma_{(\frac{1}{\sqrt{q}}, \cdots, \frac{1}{\sqrt{q}})}[f].
\]
\begin{theorem}\label{thm-generalization}
For any $a\in \BS(\C^q)$, the kernel $\Gamma_{a}[f]$ defines a bounded operator on $\ell^2(\F_q^{+})$ if and only if $f\in L^\infty(\T)$. Moreover, there exists a unitary operator $U_a: \ell^2(\F_q^{+}) \rightarrow  \bigoplus_{k=1}^\infty \ell^2(\N) $ such that for all $f\in L^\infty(\T)$, we have
\begin{align}\label{gamma-direct-gamma}
\Gamma_{a}[f]  = U_a^{-1}
\left[
\begin{array}{cccc}
\Gamma_1[f] & & &
\\
& \Gamma_1[f]& &
\\
& & \Gamma_1[f] &
\\
& & &  \ddots
\end{array}
 \right] U_a.
\end{align}
In particular, for any $a\in \BS(\C^q)$, the operators $\Gamma_q[f]$ and $\Gamma_{a}[f]$ are simultaneously unitarily equivalent for all $f\in L^\infty(\T)$.
\end{theorem}

The following result is a generalization of Theorem \ref{thm-QW-I}.
\begin{theorem}\label{thm-gen-pos}
Let $a\in \BS(\C^q)$. Then the kernel $\Gamma_{a}[f]$ is positive definite if and only if the formal Fourier series $f$ represents a positive Radon measure on $\T$.
\end{theorem}

We also have the following operator-valued results. Let $A = (A_1, \cdots, A_q)$ be a $q$-tuple of operators on a Hilbert space $\mathcal{H}$. Define for any non-empty word $w = s_{i_1} s_{i_2} \cdots s_{i_n} \in \F_q^{+}$,
\[
[w](A): = A_{i_1} A_{i_2} \cdots A_{i_n} \in B(\mathcal{H}).
\]
For the empty word $e\in \F_q^{+}$,  set $[e](A)=Id \in B(\mathcal{H})$. Then for any formal Fourier series $f$ given in \eqref{fourier-series}, we can define an operator-valued kernel $\Gamma_{A}[f]: \F_q^{+} \times \F_q^{+}\rightarrow B(\mathcal{H})$ by
\begin{align}\label{def-gamma-a-f-op}
\Gamma_{A}[f] (u, v): = \left\{
\begin{array}{lc}
 \widehat{f}(|u| - |v|) \cdot [w](A), &   \text{if $u = w v$}
\vspace{2mm}
\\
 \widehat{f}(|u| - |v|)  \cdot ([w](A))^*, &    \text{if $v = w u$}
\vspace{2mm}
\\
0, & \text{otherwise}
\end{array}
\right..
\end{align}

We refer to Paulsen \cite[Chapters 2, 3]{Paulsen-CB} for the notion of complete positivity for linear maps between operator systems.

\begin{theorem}\label{prop-pos-op}
Let $A = (A_1, \cdots, A_q)\in B(\mathcal{H})^q$ be a $q$-tuple such that
\begin{align}\label{A-le-1}
\Big\| \sum_{k=1}^q A_k^* A_k\Big\| \le 1.
\end{align}
Then the unital map
\[
\begin{array}{cccc}
\Gamma_A: & L^\infty(\T) & \rightarrow  &  B(\ell^2(\F_q^{+}) \otimes \mathcal{H})
\vspace{2mm}
\\
& f & \mapsto & \Gamma_A[f]
\end{array}
\]
is contractive and thus is completely positive.
\end{theorem}

\begin{corollary}\label{cor-op-pos}
Let $A = (A_1, \cdots, A_q)\in B(\mathcal{H})^q$ be a $q$-tuple satisfying \eqref{A-le-1}.  Then for any positive Radon measure $\mu$ on $\T$, the kernel $\Gamma_A[\mu]$ is positive definite.
\end{corollary}

\begin{remark}
Consider the special case $\mu = \delta_0$, where $\delta_0$ stands for the Dirac mass at $0 \in \R/2 \pi \Z$. Popescu \cite[Cor. 2.2]{Pop-1996} shows that $\Gamma_A[\delta_0]$ is positive definite if and only if \eqref{A-le-1} is satisfied.  Corollary \ref{cor-op-pos} provides an alternative proof of the sufficiency of the condition \eqref{A-le-1} for the positive definiteness of $\Gamma_A[\delta_0]$. On the other hand,  the necessity of the condition \eqref{A-le-1} for the positive definiteness of $\Gamma_A[\delta_0]$ can be obtained directly as follows: suppose that $\Gamma_A[\delta_0]$ is positive definite, then so is its restriction on the finite set $(\{e, s_1, \cdots, s_q\})^2$. That is,  the  operator-coefficient matrix $\left[
\begin{array}{cc}
I & \alpha
\\
\alpha^* & I_q
\end{array}
\right]$ is positive definite,  where $\alpha = (\Gamma_A[\delta_0](e, s_1), \cdots, \Gamma_A[\delta_0](e, s_q) )  = (A_1^*, \cdots, A_q^*)$ and $I_q = \mathrm{diag}(Id, \cdots, Id)$.  Then by using the equality
\[
 \left[
\begin{array}{cc}
Id & -\alpha
\\
0 & I_q
\end{array}
\right]  \left[
\begin{array}{cc}
Id & \alpha
\\
\alpha^* & I_q
\end{array}
\right] \left[
\begin{array}{cc}
Id & 0
\\
- \alpha^*  & I_q
\end{array}
\right] =
\left[
\begin{array}{cc}
Id - \alpha \alpha^* & 0
\\
0 & I_q
\end{array}
\right],
\]
we obtain that $Id - \alpha \alpha^*$ is positive definite, which implies \eqref{A-le-1}.
\end{remark}

\subsection{Branching-Toeplitz matrices}\label{sec-finite-dim}
The equality \eqref{norm-2-side} can be interpretated as
\begin{align}\label{gamma-q-1}
\| \Gamma_q[f]\|  = \| \Gamma_1[f]\|=\| T(f)\|, \quad \forall f \in L^\infty(\T).
\end{align}
It  is then  natural to wonder whether the equality \eqref{gamma-q-1} has a finite dimensional analogue. More precisely, fix any integer $n\ge 1$. Recall the definition \eqref{B-n-S-n} of the finite subset $\B_n(T_q) \subset T_q$ and consider the truncated kernel $\Gamma_q^{(n)}[f]$ on $\B_n(T_q)$:
\begin{align}\label{trunc-toep}
\Gamma_q^{(n)}[f]:=\Gamma_q[f]\big|_{\mathbb{B}_n(T_q)\times\mathbb{B}_n(T_q)}.
\end{align}
Recall also the standard Toeplitz matrix $T_n(f)$ of size $(n+1)\times (n+1)$ defined by
\begin{align}\label{Toep-rep}
T_n(f) :   = \Gamma_1^{(n)}[f] =  \big[ \widehat{f}(k-l)\big]_{0\leq k, l\leq n}.
\end{align}

\begin{question}
For all $f\in L^\infty(\T)$,  do we have
\begin{align}\label{finite-norm-eq}
\| \Gamma_q^{(n)}[f]\|  = \| \Gamma_1^{(n)}[f] \| =  \| T_n(f)\|?
\end{align}
\end{question}

\begin{remark}
The norm of a Toeplitz matrix $T_n(f)$ can be estimated by a distance  $d_n(f)$ introduced in   Nikolski \cite[Thm 5.2.1, p. 247]{Nikolski-20}, but its precise value  is not clear in general (see \cite[Section 5.2.1, p. 246]{Nikolski-20}).
\end{remark}

We shall see (in Corolloary \ref{cor-dec-finite} below) that the  one-sided inequality always holds:
\begin{align}\label{stan-le-gamma}
\| T_n(f)\| \le \| \Gamma_q^{(n)}[f]\|.
\end{align}
However, the equality \eqref{finite-norm-eq} does not hold in general.
\begin{example}\label{ex-counter}
Let $q = 2$ and $n = 1$.  Consider the function $f  = b + a e^{i \theta} - a e^{-i \theta}$ with $a> 0, b> 0$ and $a^2 + b^2 =1$.   Then
\[
\Gamma_1^{(1)} [f]  = \left[
\begin{array}{cc}
b & a
\\
-a & b
\end{array}
\right] \an \Gamma_2^{(1)}[f] = \left[
\begin{array}{ccc}
b & \frac{a}{\sqrt{2}} & \frac{a}{\sqrt{2}}
\\
-\frac{a}{\sqrt{2}} & b & 0
\\
-\frac{a}{\sqrt{2}} & 0 & b
\end{array}
\right],
\]
where  $\Gamma_2^{(1)}[f]$ is written as a $3 \times 3$ matrix by identifying $\B_1(T_2)$ with $\{0, 1, 2\}$.
By a direct computation, we have the following strict inequality
\[
\| \Gamma_2^{(1)}[f]\| = \sqrt{1 + \frac{a^2}{2}} > 1 = \| \Gamma_1^{(1)}[f]\|.
\]
\end{example}

On the other hand, fix an integer $n\ge 1$, the equality \eqref{finite-norm-eq} holds in the following cases:
\begin{itemize}
\item Case (A1): the coefficients $\widehat{f}(k)\ge 0$ for all integers $k$ with $|k|\le n$.
\item Case (A2): the finite Toeplitz matrix $T_n(f)$ (and thus  the finite branching-Toeplitz matrix $\Gamma_q^{(n)}[f]$) is Hermitian.
\item Case (A3): the symbol $f$ belongs to the Hardy space $H^\infty(\T)$.
\end{itemize}

\begin{theorem}\label{thm-con-ext}
Fix an integer $n\ge 1$.  Then  the equality \eqref{finite-norm-eq} holds in any of the three cases (A1), (A2) or (A3).
\end{theorem}

In case (A1), Theorem \ref{thm-con-ext} will be proved directly  for a fixed integer $n \ge 1$ without using  any result on the infinite dimensional branching-Toeplitz operators.   While  in  case (A2) and case (A3), the proofs of Theorem \ref{thm-con-ext} rely not only on the classical results  on Carath\'eodory-Toeplitz extension problem  (see Theorem \ref{thm-ext-pos-Toep}),  Carath\'eodory-F\'ejer-Schur extension problem   (see Theorem \ref{thm-CF}) but also on  our previous result (Theorem \ref{thm-all-symbol}) on infinite dimensional branching-Toeplitz operators. 

From our proof of Theorem \ref{thm-con-ext} (in particular, see Lemma \ref{lem-inv-sub} and  Corollary \ref{cor-dec-finite} below), we also obtain the following corollaries. For any integer $n\ge 1$, define the space of radial vectors in $\C^{\B_n(T_q)}$ by
\begin{align}\label{finite-rad-sub}
\C_{\rad}^{\B_n(T_q)}: = \Big\{v = (v_\sigma)_{\sigma\in \B_n(T_q)} \in \C^{\B_n(T_q)}\Big| \text{$v_\sigma = v_\tau$ whenever $| \sigma| = | \tau|$}\Big\}.
\end{align}

\begin{corollary}\label{cor-oth-comp-finite}
Fix an integer $n\ge 1$.  In case (A1) or case (A2) or case (A3), we have
\[
\left \| \Gamma_q^{(n)}[f]\Big|_{\C^{\B_n(T_q)} \ominus \C^{\B_n(T_q)}_{\rad}} \right\| \le  \left \| \Gamma_q^{(n)}[f]\Big|_{\C^{\B_n(T_q)}_{\rad}} \right\|.
\]
\end{corollary}

 Recall that for a bounded operator  $A$ on $\mathcal{H}$,  a vector $v \in \mathcal{H}$ is called  a {\it norming vector} of $A$ if it satisfies
\[
\text{$\|v\| = 1$ and $\| Av\| = \| A\|$.}
\]
By compactness argument, any finite matrix  admits a norming vector. 

\begin{corollary}\label{cor-norming-vec}
Fix an integer $n\ge 1$.  In case (A1) or case (A2) or case (A3), there exists a norming vector for the matrix $\Gamma_q^{(n)}[f]$ which is a radial vector.
\end{corollary}

Our proof of Theorem \ref{thm-con-ext}  has the following by-product. For fixed integer $n\ge 1$ and formal Fourier series $f$,  consider the constant concerning the generalized Carath\'eodory-F\'ejer-Schur extension problem:
\[
c_{n}(f)  := \inf\Big\{\|g\|_{\infty}\Big|   \text{$g \in L^\infty(\T)$ and $ \widehat{g}(k) = \widehat{f}(k)$ for all integers $k$ with $-n \le k \le n$} \Big\}.
\]
Nikolskaya and Farforovskaya \cite[Thm 2.11]{NF-2003} prove that
\begin{align}\label{1-3-control}
\|T_n(f)\|\leq c_{n}(f)\leq 3\|T_n(f)\|.
\end{align}
For more on the generalized Carath\'eodory-F\'ejer-Schur extension problem, one can consult \cite{BT-2001, Volberg-2004, CS-2012, Bessonov-2015}. 

We have the following improvement on the lower bound of $c_n(f)$. Note that by  the one-sided inequality  \eqref{stan-le-gamma} and Example \ref{ex-counter}, Proposition \ref{prop-low-bdd} indeed provides an improvement  of the lower bound of $c_n(f)$ in general.
\begin{proposition}\label{prop-low-bdd}
For any integer $n\ge 1$ and any formal Fourier series $f$, we have 
\[
c_{n}(f) \ge \sup_{q \ge 2} \| \Gamma_q^{(n)}[f]\|.
\]
\end{proposition}

\subsection{Application: branching-type stationary determinantal point processes}
Recall the definition of branching-type strong stationary stochastic process (abbr. branching-type SSSP) on $T_q$ introduced in  \cite{QW-I}. Set
\[
\partial T_q : = \Big\{\xi\Big|\text{$\xi$ is a geodesic ray in $T_q$ starting from the root}\Big\}.
\]
Clearly, each $\xi \in \partial T_q$, as a subset of $T_q$, can be canonically identified with the set $\N$.

\begin{definition}[Branching-type SSSP]
A stochastic process $(X_\sigma)_{\sigma \in T_q}$ on $T_q$ is called a branching-type SSSP if
\begin{itemize}
\item Restricted on every rooted geodesic ray $\xi \in \partial T_q$, we have a classical strong stationary  stochastic process $X^\xi: = (X_\sigma)_{\sigma \in \xi}$, by identifying canonically the subset $\xi\subset T_q$ with the set $\N$.
\item The family of these strong stationary stochastic processes $X^\xi$ share a common distribution. That is, for any pair $(\xi, \xi')$ of rooted geodesic rays, by using the natural identifications $\xi \simeq \N \simeq \xi'$, we have
\[
(X_\sigma)_{\sigma \in \xi} \stackrel{d}{=}  (X_{\sigma'})_{\sigma'\in \xi'}.
\]
\item For any pair of non-comparable vertices $\sigma, \tau \in T_q$, the random variables $X_\sigma, X_\tau$ are independent.
\end{itemize}
\end{definition}

 A complete classification of  Gaussian branching-type SSSP on $T_q$ is given in \cite{QW-I}. As an application of the previous results of the  present paper, we  construct a family of determinantal point processes (considered as stochastic processes of 0-1 valued random variables) which are branching-type SSSP on $T_q$ and are invariant under the action of  the group of automorphisms of $T_q$. For the background of determinantal point processes, the reader is referred to \cite{DPP-M, DPP-S,ST-palm2, ST-palm}.

\begin{proposition}\label{thm-BSSSP-DPP}
For any measurable function $f: \T \rightarrow [0, 1]$, the branching-Toeplitz kernel $\Gamma_q[f]$ represents a positive definite contractive operator on $\ell^2(T_q)$ and thus induces a determinantal point process on $T_q$ which is a branching-type SSSP and is invariant under the natural action of the group of automorphisms of $T_q$.
\end{proposition}

\section{Branching-Toeplitz operators}

\subsection{Outline of the proofs}

We shall actually present two proofs of Theorem \ref{thm-all-symbol}, both of which rely on the positive definite criterion of the kernel $\Gamma_q[f]$ recalled in Theorem \ref{thm-QW-I}.

The first proof of Theorem \ref{thm-all-symbol} is outlined as follows.

\begin{itemize}
\item  {\bf Step 1}:  We first prove that if $\Gamma_q[f]$ is a bounded operator, then  $f\in L^\infty(\T)$ and
\begin{align}\label{low-bdd}
 \|f \|_\infty \le \|\Gamma_q[f]\|.
\end{align}
 For this purpose, we introduce  in \eqref{rad-subsp} a common invariant subspace  of {\it radial vectors}   $\ell^2(T_q)_\rad \subset \ell^2(T_q)$  for all operators in $B_\TT(\ell^2(T_q))$ and show that the restriction of any bounded operator $\Gamma_q[f]$ on the subspace $\ell^2(T_q)_\rad$ is unitarily equivalent to the standard Toeplitz operator on Hardy space with the same symbol $f$ and thus, by using the results on  standard Toeplitz operators,  we are able to conclude the proof of the inequality \eqref{low-bdd}.

Note that the inequality \eqref{low-bdd} for real valued functions $f\in L^\infty(\T)$ can also be obtained using Theorem \ref{thm-QW-I}.  See the Appendix of this paper.
\item  {\bf Step 2:} We then prove in Proposition \ref{prop-norm-eq} that for any {\it real-valued} bounded function $f$ on $\T$,  the  reverse inequality
\begin{align}\label{reverse-ineq}
\|\Gamma_q[f]\| \le \|f \|_\infty
\end{align}
holds. The proof of \eqref{reverse-ineq} for real-valued  function $f\in L^\infty(\T)$ relies on the positive definite criterion of the branching-Toeplitz kernels recalled in Theorem \ref{thm-QW-I}.

As a consequence,  see Corollary \ref{cor-set-eq} below,  by decomposing a complex-valued function into its real and imaginary parts, we obtain that any $f\in L^\infty(\T)$ induces a bounded $\Gamma_q[f]$ and  thus complete the proof of the set-theoretical equality \eqref{set-th-eq}.
\item {\bf Step 3:} By  Step 1 and Step 2, the equality \eqref{norm-2-side} holds for all real-valued function $f \in L^\infty(\T)$. It remains to prove that the equality \eqref{norm-2-side} can be extended to general complex-valued symbols $f \in L^\infty(\T)$. We shall prove in Proposition \ref{prop-norm-poly} that the equality \eqref{norm-2-side} holds for any {\it analytic trigonometric polynomial} $Q \in \C[e^{i\theta}]$. Our proof relies on 
the following {\it multiplicativity property} obtained in Lemma \ref{lem-mul}:
\[
(\Gamma_q[Q])^* \Gamma_q[Q]  = \Gamma_q[\bar{Q}] \Gamma_q[Q] =  \Gamma_q[|Q|^2],
\]
which is a special case of the following
\begin{align}\label{mul-P-Q}
\Gamma_q[P] \Gamma_q[Q]   =  \Gamma_q[PQ], \quad \forall P \in \C[e^{i \theta}, e^{-i \theta}], \forall Q \in \C[e^{i\theta}].
\end{align}
\item {\bf Step 4:} Finally, to complete the proof of  the equality \eqref{norm-2-side} for all $f\in L^\infty(\T)$, we show in Proposition \ref{prop-tri-symbol} that for any trigonometric polynomial $P \in \C[e^{i \theta}, e^{-i \theta}]$, by writing $P  = e^{-iN\theta} (e^{iN\theta} P)$ with $e^{iN\theta} P\in \C[e^{i\theta}]$ for large enough integer $N\ge 1$, we may apply \eqref{mul-P-Q}  to get
\[
\Gamma_q[P] = \Gamma_q[e^{-iN\theta}] \Gamma_q[e^{i N\theta} P]  = (\Gamma_q[e^{iN\theta}])^{*}  \Gamma_q[e^{i N\theta} P]
\]
and
\[
\|\Gamma_q[P]\|\le \| P\|_\infty.
\]
The case of general $f \in L^\infty(\T)$ can then be obtained using a standard  approximation argument with F\'ejer kernels. See \S \ref{sec-gen-symbol} for more details.
\end{itemize}

\bigskip

The second proof of Theorem \ref{thm-all-symbol} differs from the first one in that it replaces the above Step 3 and Step 4 by the following Step 3', where we combine our Theorem \ref{thm-QW-I} with the well-known Russo-Dye Theorem.  Recall that we call a linear map $\phi: \mathcal{A}\rightarrow \mathcal{B}$ between two $C^*$-algebras  positive if it sends all positive elments of $\mathcal{A}$ to positive elements of $\mathcal{B}$.

\begin{theorem}[{Russo-Dye, \cite{Russo-Dye} or \cite[Cor. 2.9]{Paulsen-CB}}]
Let $\mathcal{A}$ and $\mathcal{B}$ be $C^*$-algebras with unit and let $\phi: \mathcal{A} \rightarrow \mathcal{B}$ be a linear positive map. Then $\| \phi\| = \| \phi(1_{\mathcal{A}})\|$.
\end{theorem}

The details of the second proof of Theorem \ref{thm-all-symbol} is given as follows.
\begin{itemize}
\item {\bf Steps 1 and 2:} The same as Steps 1 and 2 in the first proof of Theorem \ref{thm-all-symbol}.
\item {\bf Step 3':} By Steps 1 and 2 and also Theorem \ref{thm-QW-I},  the map $\phi: L^\infty(\T)\rightarrow B(\ell^2(T_q))$ given by $\phi(f) : = \Gamma_q[f]$ is a linear positive map between two unital $C^*$-algebras. Since $\phi$ maps the constant function $1$ to the identity operator on $\ell^2(T_q)$, by Russo-Dye theorem, we have $\| \phi\| =1$, which implies the desired reverse inequality \eqref{reverse-ineq} for general symbols $f\in L^\infty(\T)$.
\end{itemize}

\subsection{The invariant subspace of radial vectors}

 Let $q \ge 2$ be an integer and let $f$ be a formal Fourier series on $\T$ given in \eqref{fourier-series}. Recall the definition \eqref{def-T-f} of the branching-Toeplitz kernel $\Gamma_q[f]$ on $T_q$,  the definition \eqref{B-n-S-n} of the finite subset $\B_n(T_q)$ and the definition \eqref{trunc-toep} of the truncated kernel $\Gamma_q^{(n)}[f]$ on $\B_n(T_q)$. Recall also the definition \eqref{finite-rad-sub} of the space $\C_{\rad}^{\B_n(T_q)}$ of radial vectors.  A natural orthonormal basis of $\C_{\mathrm{rad}}^{\mathbb{B}_n(T_q)}$ is given by
\begin{align}\label{def-onb}
\bigg\{h_k=\frac{\mathds{1}_{\mathbb{S}_k}}{\sqrt{q^k}}:k=0,1,\cdots, n\bigg\}.
\end{align}
Let $P_{\rad}^{(n)}: \C^{\B_n(T_q)} \rightarrow \C^{\B_n(T_q)}_{\rad}$ denote the orthogonal projection from $\C^{\B_n(T_q)}$ onto the space of radial vectors.

\begin{lemma}\label{lem-inv-sub}
The subspace $\C^{\B_n(T_q)}_{\rad}\subset \C^{\B_n(T_q)}$ is a common invariant subspace for  $\Gamma_q^{(n)}[f]$ when $f$ ranges over all formal Fourier series. Moreover, with respect to the orthonormal basis \eqref{def-onb} of $\C_\rad^{\B_n(T_q)}$, the operator $ P_{\rad}^{(n)} \Gamma_q^{(n)}[f] P_{\rad}^{(n)}$ has the standard Toeplitz matrix representation given in \eqref{Toep-rep}.
\end{lemma}

\begin{proof}
For any $0 \le l \le n$, we have
\begin{align*}
(\Gamma_q^{(n)}[f] h_l )(\sigma)&  = \sum_{\tau \in \B_n(T_q)} (\Gamma_q^{(n)}[f]) (\sigma, \tau) \frac{\mathds{1}_{\mathbb{S}_l}(\tau)}{\sqrt{q^l}} = \frac{1}{\sqrt{q^l}} \sum_{\tau \in \mathbb{S}_l}  \sqrt{q^{- \big| |\sigma| - |\tau|\big|}} \widehat{f}( | \sigma| - |\tau|) \mathds{1}_{\mathcal{C}}(\sigma, \tau).
\end{align*}
Fix $\sigma \in \B_n(T_q)$ and we divide the computation of $(\Gamma_q^{(n)}[f] h_l )(\sigma)$ into two cases.
\begin{itemize}
\item  If $| \sigma| \ge l$, then there exists a unique $\tau \in \mathbb{S}_l$ which is comparable with $\sigma$. That is,
\[
\# \Big\{\tau \in \mathbb{S}_l \Big|   (\sigma, \tau) \in \mathcal{C}\Big\} =  1.
 \]
Hence
\begin{align*}
(\Gamma_q^{(n)}[f] h_l )(\sigma) = \frac{1}{\sqrt{q^l}}  \sqrt{q^{-  |\sigma|+ l}} \widehat{f}( | \sigma| - l) =   \frac{1}{\sqrt{q^{|\sigma|}}} \widehat{f}( | \sigma| - l ).
\end{align*}
\item If $| \sigma| < l$, then
\[
\# \Big\{\tau \in \mathbb{S}_l \Big|   (\sigma, \tau) \in \mathcal{C}\Big\} =  q^{l-|\sigma|}.
 \]
Hence
\begin{align*}
(\Gamma_q^{(n)}[f] h_l )(\sigma) = \frac{1}{\sqrt{q^l}}  \sqrt{q^{  |\sigma|- l}} \widehat{f}( | \sigma| - l ) q^{l-|\sigma|} =   \frac{1}{\sqrt{q^{|\sigma|}}} \widehat{f}( | \sigma| - l) .
\end{align*}
\end{itemize}
Therefore, we obtain
\begin{eqnarray*}
\Gamma_q^{(n)}[f]h_l&=&\sum_{k=0}^n   \frac{1}{\sqrt{q^k}} \widehat{f}( k - l ) 1_{\mathbb{S}_k}= \sum_{k =0}^n \widehat{f}(k-l) h_k.
\end{eqnarray*}
This shows not only that $\C_\rad^{\B_n(T_q)}$ is an invariant subspace for the operator $\Gamma_q^{(n)}[f]$ but also that the operator $P_\rad^{(n)}\Gamma_q^{(n)}[f] P_\rad^{(n)}$ has the Toeplitz matrix representation \eqref{Toep-rep}  with respect to the orthonormal basis \eqref{def-onb}.
\end{proof}

\begin{corollary}\label{cor-dec-finite}
Fix an integer $n\ge 1$.  Then with respect to the orthogonal decomposition
\[
\C^{\B_n(T_q)}=\C^{\mathbb{B}_n(T_q)}_{\mathrm{rad}}\oplus \Big( \C^{\B_n(T_q)} \ominus \C^{\mathbb{B}_n(T_q)}_{\mathrm{rad}}\Big),
\]
any branching-Toeplitz matrix $\Gamma_q^{(n)}[f]$, considered as a linear operator on $\C^{\B_n(T_q)}$,  has the following block representation:
\begin{align}\label{BT-block-form-finite}
\Gamma_q^{(n)}[f] =
\left[
\begin{array}{cc}
 \Gamma_q^{(n)}[f]\Big|_{\C^{\mathbb{B}_n(T_q)}_{\mathrm{rad}}} &  0
\\
0 &  \Gamma_q^{(n)}[f]\Big|_{ \C^{\B_n(T_q)} \ominus \C^{\mathbb{B}_n(T_q)}_{\mathrm{rad}}}
\end{array}
\right].
\end{align}
Moreover, we have
\begin{align}\label{norm-sup-2}
\begin{split}
\| \Gamma_q^{(n)}[f]\| & = \max  \left( \left\| \Gamma_q^{(n)}[f]\Big|_{\C^{\mathbb{B}_n(T_q)}_{\mathrm{rad}}}\right\|,\quad \left\| \Gamma_q^{(n)}[f]\Big|_{ \C^{\B_n(T_q)} \ominus \C^{\mathbb{B}_n(T_q)}_{\mathrm{rad}}}\right\| \right)
\\
&  = \max  \left( \left\| T_n(f) \right\|, \quad \left\| \Gamma_q^{(n)}[f]\Big|_{ \C^{\B_n(T_q)} \ominus \C^{\mathbb{B}_n(T_q)}_{\mathrm{rad}}}\right\| \right).
\end{split}
\end{align}
\end{corollary}
\begin{proof}
Note that $(\Gamma_q^{(n)}[f])^{*} = \Gamma_q^{(n)}[\bar{f}]$. By Lemma \ref{lem-inv-sub}, the space $\C^{\mathbb{B}_n(T_q)}_{\mathrm{rad}}$ is invariant under both $\Gamma_q^{(n)}[f]$ and $(\Gamma_q^{(n)}[f])^*$. It follows that $\C^{\B_n(T_q)} \ominus \C^{\mathbb{B}_n(T_q)}_{\mathrm{rad}}$ is also invariant  under $\Gamma_q^{(n)}[f]$ and we have the block representation \eqref{BT-block-form-finite}. The equality  \eqref{norm-sup-2}  follows from the equality
\[
\left\| \Gamma_q^{(n)}[f]\Big|_{\C^{\mathbb{B}_n(T_q)}_{\mathrm{rad}}}\right\| = \| T_n(f)\|,
\]
which is also a consequence of  Lemma \ref{lem-inv-sub}.
\end{proof}

\begin{lemma}\label{lem-one-side}
If the branching-Toeplitz operator $\Gamma_q[f]:\ell^2(T_q)\to \ell^2(T_q)$ is bounded, then the formal Fourier series \eqref{fourier-series} defines a function  in $L^\infty(\mathbb{T})$ and we have
\begin{align}\label{sup-norm}
 \|f \|_\infty \le \|\Gamma_q[f]\|.
\end{align}
\end{lemma}
\begin{proof}
Let $T_{n}(f)$ denote the standard $(n+1)\times (n+1)$-Toeplitz matrix in \eqref{Toep-rep}. By the classical result on Toeplitz matrices (see, e.g., Peller \cite[Thm 1.1, p. 88]{Peller-book}),  we have
\[
\|f \|_\infty  = \sup_{n\in \N} \| T_{n}(f)\|.
\]
An immediate consequence of Lemma \ref{lem-inv-sub} is
\[
\| T_{n}(f)\| = \| P_\rad^{(n)} \Gamma_q^{(n)}[f] P_{\rad}^{(n)}\| \le \| \Gamma_q^{(n)}[f]\| \le \| \Gamma_q[f]\|,
\]
this completes the proof of the lemma.
\end{proof}

Introduce the space of  radial vectors:
\begin{align}\label{rad-subsp}
\ell^2(T_q)_{\rad}: = \Big\{v = (v_\sigma)_{\sigma\in T_q} \in \ell^2(T_q)\Big| \text{$v_\sigma = v_\tau$ whenever $| \sigma| = | \tau|$}\Big\}.
\end{align}
A natural orthonormal basis of $\ell^2(T_q)_{\rad}$ is given by
\[
\bigg\{h_k=\frac{\mathds{1}_{\mathbb{S}_k}}{\sqrt{q^k}} \Big| k=0,1, 2, \cdots\bigg\}, \quad \text{where $\mathbb{S}_k:  = \{ \sigma \in T_q \big| | \sigma |  = k \}$.}
\]

Let $P_{\rad}: \ell^2(T_q) \rightarrow \ell^2(T_q)_{\rad}$ denote the associated orthogonal projection.

\begin{proposition}\label{prop-bdd}
The space $\ell^2(T_q)_{\rad}$ is a common invariant subspace for all operators in $B_\TT(\ell^2(T_q))$. Moreover,  for all $f\in L^\infty(\T)$, we have the following commutative diagram
\[
\begin{CD}
\ell^2(T_q)_{\rad} @> P_{\rad} \circ \Gamma_q[f] \circ P_{\rad}>>   \ell^2(T_q)_{\rad}
\\
@V{\Phi}V{\simeq }V     @V{\Phi}V{\simeq}V
\\
 \ell^2(\N) @>\qquad \Gamma_1[f] =  T(f)\qquad >> \ell^2(\N)
\end{CD},
\]
where $\Phi$ is the unitary operator defined by
$\Phi h_k = \delta_k$ for all $k \in \N$
and $T(f)$ is the operator  on $\ell^2(\N)$ corresponding to the standard Toeplitz matrix given in \eqref{cla-toep}.
\end{proposition}

\begin{proof}
The proof of Proposition \ref{prop-bdd} is similar to that of Lemma \ref{lem-inv-sub}.
\end{proof}

\subsection{Hermitian branching-Toeplitz operators}

In this section, we will give the proof of the equality \eqref{norm-2-side} for Hermitian kernels $\Gamma_q[f]$. As a consequence, we obtain
\[
\textit{$\Gamma_q[f]$ is bounded if and only if $f\in L^\infty(\T)$.}
\]

\bigskip

Note first that  $\Gamma_q[f]$ is Hermitian if and only if the formal Fourier series $f$ defined in \eqref{fourier-series} is Hermitian in the sense that
\[
\widehat{f}(-n) =  \overline{\widehat{f}(n)}, \quad \forall n \in \Z.
\]
 Let $L^\infty(\T; \R) \subset L^\infty(\T)$ denote the subset of bounded real-valued functions.

\begin{proposition}\label{prop-norm-eq}
Assume that the formal Fourier series $f$ is Hermitian.  Then $\Gamma_q[f]$ induces a bounded operator on $\ell^2(T_q)$ if and only if the formal Fourier series $f$ represents a function in $L^\infty(\T; \R)$. Moreover, in this case,  we have
\begin{align}\label{norm-eq}
 \| \Gamma_q[f]\|  = \| f \|_\infty.
\end{align}
\end{proposition}

\begin{proof}
By Lemma \ref{lem-one-side}, it remains  to show that if $f \in L^\infty(\T; \R)$, then $\Gamma_q[f]$ is bounded on $\ell^2(T_q)$ with norm $\| \Gamma_q[f]\|\le \| f \|_\infty$.

Suppose that  $f \in L^\infty(\T; \R)$ and write
$
\lambda =\|f\|_\infty.
$
Since $f$ is real-valued, we have
$$\left\{
    \begin{array}{ll}
     g: =  f+ \lambda \geq 0
\vspace{2mm}
\\
     h: =  \lambda - f \geq 0
    \end{array}
  \right..
$$
Therefore, by Theorem \ref{thm-QW-I}, the corresponding branching-Toeplitz kernels $\Gamma_q[g]$ and $\Gamma_q[h]$ are both positive definite. Clearly, we have
\[
\Gamma_q[g] = \Gamma_q[f] + \lambda \cdot Id, \quad \Gamma_q[h] = \lambda \cdot Id - \Gamma_q[f],
\]
where $Id$ denotes the kernel corresponding to the identity operator on $\ell^2(T_q)$.  Thus  for any $n\ge 1$, by truncating onto the finite subset $\B_n(T_q) \subset T_q$ as in \eqref{trunc-toep} and by denoting $Id_{\B_n(T_q)}$ the identity matrix on $\C^{\B_n(T_q)}$, we prove that the following two finite  matrices
\[
\Gamma_q^{(n)}[g] = \Gamma_q^{(n)}[f] + \lambda \cdot Id_{\B_n(T_q)}, \quad \Gamma_q^{(n)}[h] = \lambda \cdot Id_{\B_n(T_q)} - \Gamma_q^{(n)}[f],
\]
are both positive definite. This implies  $\| \Gamma_q^{(n)}[f]\|\le \lambda$.   Therefore,
\[
\| \Gamma_q[f]\| = \sup_{n \in \N }\| \Gamma_q^{(n)}[f]\|\le \lambda = \| f \|_\infty
\]
and we complete the proof.
\end{proof}

\begin{corollary}\label{cor-set-eq}
For any integer $q \ge 2$, we have the following set-theoretical equality
\[
B_\TT(\ell^2(T_q)) = \Big\{ \Gamma_q[f]\Big| f \in L^\infty(\T) \Big\}.
\]
\end{corollary}

\begin{proof}
By Lemma \ref{lem-one-side}, it remains to show that  for any $f\in L^\infty(\T)$, we have
$\Gamma_q[f] \in B(\ell^2(T_q))$. Indeed, if $f \in L^\infty(\T)$,
 then  $f + \bar{f}$ and  $i (f - \bar{f})$ are in $L^\infty(\T; \R)$. Therefore,  by Proposition \ref{prop-norm-eq}, both kernels $\Gamma_q[f + \bar{f}]$ and  $\Gamma_q[ i(f - \bar{f})]$ induce bounded operators on  $\ell^2(T_q)$. Hence
\[
\Gamma_q[f] =  \Gamma_q\Big[ \frac{f + \bar{f}}{2}\Big] +  i  \Gamma_q\Big[ \frac{f - \bar{f}}{2i}\Big]  \in B(\ell^2(T_q)).
\]
\end{proof}

\subsection{Multiplicativity}
In this section, we will prove Theorem \ref{thm-gen-mul}.

Let $\C[e^{i\theta}, e^{-i \theta}]$ be the set of trigonometric polynomials on $\T$ and let $\C[e^{i\theta}]$ be the subset of analytic trigonometric polynomials. That is,
\begin{align*}
\C[e^{i\theta}, e^{-i \theta}] &= \Big\{P = \sum_{|n|\le N} a_n e^{in \theta} \Big| N \in \N, a_n \in \C\Big\},
\\
\C[e^{i\theta}] & = \Big\{Q= \sum_{n  = 0}^N a_n e^{in \theta} \Big| N \in \N, a_n \in \C\Big\}.
\end{align*}

\begin{lemma}\label{lem-mul}
For any $P\in \C[e^{i\theta}, e^{-i \theta}]$ and any $Q \in \C[e^{i\theta}]$, we have
\[
\Gamma_q[P] \Gamma_q[Q] = \Gamma_q[PQ].
\]
\end{lemma}

\begin{proof}
By linearity, we only need to show that for any $m \in \Z$ and any $n \in \N$, the following equality holds:
\begin{align}\label{mul-property}
\Gamma_q[e^{i m \theta}] \Gamma_q[e^{in \theta}]  = \Gamma_q[e^{i (m + n)\theta}].
\end{align}

Recall  by definition, we have
\[
\Gamma_q[e^{i (m + n) \theta}](\sigma_1, \sigma_2)   = q^{- \frac{d(\sigma_1, \sigma_2)}{2}} \mathds{1} (| \sigma_1| - |\sigma_2| = m +n) \cdot \mathds{1}_{\mathcal{C}}(\sigma_1, \sigma_2).
\]

{\flushleft  Case (i): $m < 0$.}  Fix any $\sigma_1, \sigma_2 \in T_q$. We shall use the elementary fact that for any $\tau \in T_q$, the set of ancestors of $\tau$ is a totally ordered set with respect to $\pless$. That is, the following set is a linear set for the partial order $\pless$:
\[
Ancestor(\tau): = \{\sigma\in T_q | \sigma \pless \tau\}.
\]
Therefore, if there exists $\tau \in T_q$ such that
\begin{align}\label{tau-unique-bis}
\left\{
\begin{array}{l}
(\sigma_1, \tau) \in \mathcal{C} \an  (\tau, \sigma_2) \in \mathcal{C}
\vspace{2mm}
\\
|\sigma_1| - | \tau| = m < 0 \an |\tau |  - | \sigma_2| = n \ge 0,
\end{array}
\right.
\end{align}
then $\sigma_1, \sigma_2 \in Ancestor(\tau)$ and
\begin{align}\label{ancestor-conseq}
\sigma_1, \sigma_2 \in \mathcal{C}  \an |\sigma_1| - | \sigma_2| = n + m.
\end{align}
Moreover, if we define
\[
\mathrm{Admissible}(\sigma_1, \sigma_2; m, n): = \Big\{ \tau \in T_q \Big| \text{$\tau$ satisfies the condition \eqref{tau-unique-bis}}\Big\},
\]
then for any pair $(\sigma_1, \sigma_2)$ of vertices in $T_q$ satisfying \eqref{ancestor-conseq}, we have
\[
\#  \mathrm{Admissible}(\sigma_1, \sigma_2; m, n) = q^{\min(-m, n)}.
\]
Combining the above arguments, we obtain
\begin{align*}
& (\Gamma_q[e^{i m \theta}] \Gamma_q[e^{i n \theta}])(\sigma_1, \sigma_2) = \sum_{\tau \in T_q}  \Gamma_q[e^{i m \theta}] (\sigma_1, \tau) \Gamma_q[e^{i n \theta}] (\tau, \sigma_2)
\\
&   = \sum_{\tau \in T_q} q^{- d(\sigma_1, \tau)/2} \mathds{1} (| \sigma_1| - |\tau| = m) \mathds{1}_{\mathcal{C}}(\sigma_1, \tau) \cdot q^{- d(\sigma_2, \tau)/2} \mathds{1} (|\tau| -| \sigma_2|  = n) \mathds{1}_{\mathcal{C}}( \tau, \sigma_2)
\\
& = \sum_{\tau \in \mathrm{Admissible}(\sigma_1, \sigma_2; m, n)}  q^{m/2} q^{-n/2}  =  q^{(m-n)/2} \# \mathrm{Admissible}(\sigma_1, \sigma_2; m, n)
\\
& = q^{(m-n)/2} q^{\min(-m, n)} \mathds{1} (\text{the pair $(\sigma_1, \sigma_2)$ satisfies the condition \eqref{ancestor-conseq}})
\\
& = q^{- \frac{|m+n|}{2}}  \mathds{1}(|\sigma_1| - |\sigma_2| = n +m) \cdot \mathds{1}_\mathcal{C}(\sigma_1, \sigma_2)
\\
& = q^{- \frac{d(\sigma_1, \sigma_2)}{2}}  \mathds{1}(|\sigma_1| - |\sigma_2| = n +m) \cdot \mathds{1}_\mathcal{C}(\sigma_1, \sigma_2).
\end{align*}
Hence we complete the proof of the equality \eqref{mul-property}  in the first case.

{\flushleft Case (ii): $m \ge 0$.}  Fix any $\sigma_1, \sigma_2 \in T_q$. Note that if there exists $\tau \in T_q$ such that
\begin{align}\label{tau-unique}
\left\{
\begin{array}{l}
(\sigma_1, \tau) \in \mathcal{C} \an  (\tau, \sigma_2) \in \mathcal{C}
\vspace{2mm}
\\
|\sigma_1| - | \tau| = m \ge 0 \an |\tau |  - | \sigma_2| = n \ge 0,
\end{array}
\right.
\end{align}
then such $\tau$ is unique and we have $\sigma_2 \pless \tau \pless \sigma_1$ and
\begin{align}\label{comp-2-sigma}
 (\sigma_1, \sigma_2)\in \mathcal{C}\an |\sigma_1| - |\sigma_2| = m + n.
\end{align}
On the other hand, if a pair $(\sigma_1, \sigma_2)$ of vertices satisfies the condition \eqref{comp-2-sigma}, then there exists a unique $\tau \in T_q$ satisfying \eqref{tau-unique}.  Therefore, we obtain
\begin{align*}
& (\Gamma_q[e^{i m \theta}] \Gamma_q[e^{i n \theta}])(\sigma_1, \sigma_2) = \sum_{\tau \in T_q}  \Gamma_q[e^{i m \theta}] (\sigma_1, \tau) \Gamma_q[e^{i n \theta}] (\tau, \sigma_2)
\\
&   = \sum_{\tau \in T_q} q^{- d(\sigma_1, \tau)/2} \mathds{1} (| \sigma_1| - |\tau| = m) \mathds{1}_{\mathcal{C}}(\sigma_1, \tau) \cdot q^{- d(\sigma_2, \tau)/2} \mathds{1} (|\tau| -| \sigma_2|  = n) \mathds{1}_{\mathcal{C}}( \tau, \sigma_2)
\\
& =
\left\{
\begin{array}{ll}
 q^{- \frac{d(\sigma_1, \sigma_2)}{2}} & \text{if $ (\sigma_1, \sigma_2)\in \mathcal{C}\an |\sigma_1| - |\sigma_2| = m + n$}
\vspace{2mm}
\\
0  & \text{otherwise}
\end{array}
\right.
\\
& = q^{- \frac{d(\sigma_1, \sigma_2)}{2}} \mathds{1} (| \sigma_1| - |\sigma_2| = m +n) \cdot \mathds{1}_{\mathcal{C}}(\sigma_1, \sigma_2).
\end{align*}
Hence we complete the proof of the equality \eqref{mul-property}  in the remaining case.
\end{proof}

\begin{lemma}\label{lem-weak-star}
Fix any $\sigma_1, \sigma_2 \in T_q$ and any $\varphi \in L^\infty(\T)$. The following two maps
\[
\begin{array}{ccc}
L^\infty(\T)& \longrightarrow & \C
\\
f & \mapsto & (\Gamma_q[\varphi] \Gamma_q[f])(\sigma_1, \sigma_2)
\end{array}
\quad \an
\begin{array}{ccc}
L^\infty(\T)& \longrightarrow & \C
\\
f & \mapsto & (\Gamma_q[f]\Gamma_q[\varphi])(\sigma_1, \sigma_2)
\end{array}
\]
are continuous with respect to the weak-star topology $\sigma(L^\infty, L^1)$ on $L^\infty(\T)$.
\end{lemma}

\begin{remark}\label{rmk-w-star}
Clearly, for any $\sigma_1, \sigma_2 \in T_q$, the map
\begin{align}\label{weak-star-one-f}
\begin{array}{ccc}
L^\infty(\T)& \longrightarrow & \C
\\
f & \mapsto & (\Gamma_q[f])(\sigma_1, \sigma_2)
\end{array}
\end{align}
is continuous with respect to the weak-star topology on $L^\infty(\T)$.
\end{remark}

\begin{proof}[Proof of Lemma \ref{lem-weak-star}]
Let us prove the continuity of the map $f \mapsto (\Gamma_q[\varphi] \Gamma_q[f])(\sigma_1, \sigma_2)$ with respect to the weak-star topology on $L^\infty(\T)$. The proof for the other map is similar.
Write
\begin{align}\label{expansion-prod}
(\Gamma_q[\varphi] \Gamma_q[f])(\sigma_1, \sigma_2) = \sum_{\tau \in T_q} q^{-\frac{d(\sigma_1, \tau)}{2}} \widehat{\varphi}(|\sigma_1| - |\tau|)  \mathds{1}_\mathcal{C}(\sigma_1, \tau) \cdot q^{-\frac{d(\tau, \sigma_2)}{2}} \widehat{f}(|\tau|- | \sigma_2|) \mathds{1}_\mathcal{C}( \tau, \sigma_2).
\end{align}
Note that the non-zero contribution in \eqref{expansion-prod} is the summation  over the set
\[
S(\sigma_1, \sigma_2): = \Big\{\tau\in T_q\Big| \text{$(\sigma_1, \tau) \in \mathcal{C}$ and $(\sigma_2, \tau) \in \mathcal{C}$}\Big\}.
\]
We need to deal with the following three cases depending on the relation between the two vertices $\sigma_1, \sigma_2 \in T_q$.
{\flushleft Case (i): $\sigma_1, \sigma_2$ are non-comparable.} In this case, we have
\[
\# S(\sigma_1, \sigma_2)<\infty.
\]
Then since for any $l\in \Z$, the map $f\mapsto \widehat{f}(l)$ is continuous with respect to the weak-star topology on $L^\infty(\T)$, so is the map $f \mapsto (\Gamma_q[\varphi] \Gamma_q[f])(\sigma_1, \sigma_2)$.

{\flushleft Case (ii): $\sigma_1\pless \sigma_2$.} In this case, we have the following partion
\begin{align*}
 S(\sigma_1, \sigma_2) =
 \underbrace{\Big\{\tau\in S(\sigma_1, \sigma_2) \Big| \tau\pless \sigma_2\Big\}}_{\text{denoted $S'(\sigma_1, \sigma_2)$}} \bigsqcup \underbrace{S(\sigma_1, \sigma_2)\setminus S'(\sigma_1, \sigma_2)}_{\text{denoted $S''(\sigma_1, \sigma_2)$}}.
\end{align*}
Clearly, we have
\[
\# S'(\sigma_1, \sigma_2)<\infty.
\]
 Dividing the summation in \eqref{expansion-prod} as the summation over the finite subset $S'(\sigma_1, \sigma_2)$ and the summation over the infinite subset $S''(\sigma_1, \sigma_2)$, we obtain
\[
(\Gamma_q[\varphi] \Gamma_q[f])(\sigma_1, \sigma_2)   = \underbrace{\sum_{\tau \in S'(\sigma_1, \sigma_2)} \cdots}_{\text{denoted $\Sigma'(\varphi, f, \sigma_1, \sigma_2)$}} +  \underbrace{ \sum_{\tau \in S'(\sigma_1, \sigma_2)}  \cdots}_{\text{denoted $\Sigma''(\varphi, f, \sigma_1, \sigma_2)$}}.
\]
Then the map $f \mapsto \Sigma'(\varphi, f, \sigma_1, \sigma_2)$ is continuous with respect to the weak-star topology on $L^\infty(\T)$. For the second term, we have
\begin{align}\label{f-psi-pair}
\begin{split}
\Sigma''(\varphi, f, \sigma_1, \sigma_2) & = \sum_{k=1}^\infty q^{-d(\sigma_1, \sigma_2)/2 - k/2} \widehat{\varphi}(- d(\sigma_1, \sigma_2) - k) q^{-k/2} \widehat{f}(k)  \cdot q^k
\\
&  = q^{-\frac{d(\sigma_1, \sigma_2)}{2}} \sum_{k = 1}^\infty  \widehat{\varphi}(- d(\sigma_1, \sigma_2)-k) \widehat{f}(k)
\\
& = \int_\T  f(e^{i\theta})\underbrace{ \Big( q^{-\frac{d(\sigma_1, \sigma_2)}{2}} \sum_{k=1}^\infty   \widehat{\varphi}(- d(\sigma_1, \sigma_2)-k)  e^{-i k\theta}\Big)}_{\text{denoted $\psi(e^{i\theta})$}} \frac{d\theta}{2\pi}.
\end{split}
\end{align}
The assumption  $\varphi \in L^\infty(\T)\subset L^2(\T)$ implies $\psi\in L^2(\T)\subset L^1(\T)$. Thus the equality \eqref{f-psi-pair} shows that the map $f\mapsto \Sigma''(\varphi, f, \sigma_1, \sigma_2)$ is continuous with respect to the weak-star topology  on $L^\infty(\T)$. This completes the proof of case (ii).

{\flushleft Case (iii): $\sigma_2\pless \sigma_1$.} The proof of case (iii) is similar to that of case (ii).
\end{proof}

\begin{proof}[Proof of Theorem \ref{thm-gen-mul}]
We first show that if either $f \in H^\infty(\T)$ or $g \in H^\infty(\T)$, then we have $\Gamma_q[\bar{f} g] = \Gamma_q[\bar{f}]\Gamma_q[g]$.

{\flushleft Case (i) $g \in H^\infty(\T)$.} In this case, we may find a sequence $(f_n)_{n=1}^\infty$ in $\C[e^{i\theta}, e^{-i \theta}]$ and a sequence $(g_n)_{n=1}^\infty \in \C[e^{i\theta}]$ converging to $f$ and $g$ in $L^\infty(\T)$ respectively in the weak-star topology.   By Lemma \ref{lem-mul}, for any $n, k \ge 1$, we have
\[
\Gamma_q[\bar{f}_k g_n] = \Gamma_q[\bar{f}_k] \Gamma_q[g_n].
\]
In particular,  for any pair of vertices $\sigma_1, \sigma_2 \in T_q$, we have
\[
\Gamma_q[\bar{f}_k g_n]  (\sigma_1, \sigma_2)= \Gamma_q[\bar{f}_k] \Gamma_q[g_n] (\sigma_1, \sigma_2).
\]
Then by Lemma \ref{lem-weak-star} and Remark \ref{rmk-w-star}, we have
\begin{align*}
\Gamma_q[\bar{f}  g] (\sigma_1, \sigma_2) &  = \lim_{k\to\infty} \Big(\lim_{n\to\infty} \Gamma_q[\bar{f}_k g_n]  (\sigma_1, \sigma_2)\Big)
\\
&  =  \lim_{k\to\infty} \Big(\lim_{n\to\infty} \Gamma_q[\bar{f}_k]\Gamma_q[g_n]  (\sigma_1, \sigma_2)\Big)  = \Gamma_q[\bar{f}] \Gamma_q[g] (\sigma_1, \sigma_2).
\end{align*}
Since $\sigma_1, \sigma_2$ are chosen arbitrarily, we obtain the desired equality $\Gamma_q[\bar{f} g] = \Gamma_q[\bar{f}] \Gamma_q[g]$.

{\flushleft Case (ii) $f \in H^\infty(\T)$.} This case can be reduced to case (i). Indeed, we have
\[
\Gamma_q[\bar{f} g] = \Big(\Gamma_q[f \bar{g}]\Big)^*  \stackrel{\text{Case (i)}}{=\joinrel=\joinrel=\joinrel=} \Big(\Gamma_q[\bar{g}] \Gamma_q[f]\Big)^* = \Gamma_q[f]^* \Gamma_q[\bar{g}]^* = \Gamma_q[\bar{f}] \Gamma_q[g].
\]

Now we prove the converse direction: if $\Gamma_q[\bar{f} g] = \Gamma_q[\bar{f}]\Gamma_q[g]$, then either $f\in H^\infty(\T)$ or $g \in H^\infty(\T)$.  In fact,  since $\ell^2(T_q)_{\rad}$ is invariant for both $\Gamma_q[f]$ and $(\Gamma_q[f])^* = \Gamma_q[\bar{f}]$,  so is its orthogonal complement $\ell^2(T_q) \ominus \ell^2(T_q)_{\rad}$. Thus, with respect to the orthogonal decomposition $\ell^2(T_q) = \ell^2(T_q) \oplus \Big( \ell^2(T_q) \ominus \ell^2(T_q)_\rad\Big)$, the operators $\Gamma_q[f]$ can be represented in the following block form:
\begin{align}\label{BT-block-form}
\Gamma_q[f] =
\left[
\begin{array}{cc}
 \Gamma_q[f]\big|_{\ell^2(T_q)_{\rad}} &  0
\\
0 &  \Gamma_q[f]\big|_{\ell^2(T_q) \ominus \ell^2(T_q)_{\rad}}
\end{array}
\right].
\end{align}
Thus by Proposition \ref{prop-bdd} and the block form \eqref{BT-block-form}, the equality $\Gamma_q[\bar{f} g] = \Gamma_q[\bar{f}]\Gamma_q[g]$ implies the equality $T(\bar{f} g) = T(\bar{f})T(g)$ for standard Toeplitz operators and hence, by  Brown and Halmos \cite[Thm. 8]{BH-1963}, either $f\in H^\infty(\T)$ or $g \in H^\infty(\T)$.
\end{proof}

\subsection{General symbols}\label{sec-gen-symbol}

\begin{proposition}\label{prop-norm-poly}
For any analytic trigonometric polynomial $Q \in \C[e^{i\theta}]$, we have
\begin{align}\label{Q-norm-eq}
\|\Gamma_q[Q]\| = \| Q\|_\infty.
\end{align}
\end{proposition}

\begin{proof}
Let $Q \in \C[e^{i\theta}]$. Note that $(\Gamma_q[Q])^* = \Gamma_q[\bar{Q}]$.
By Lemma \ref{lem-mul}, we obtain
\[
(\Gamma_q[Q])^*  \Gamma_q[Q] = \Gamma_q[\bar{Q}] \Gamma_q[Q]  =  \Gamma_q[\bar{Q}Q] = \Gamma_q[|Q|^2].
\]
Since $|Q|^2 \in L^\infty(\T; \R)$, we may apply Proposition \ref{prop-norm-eq} to conclude
\[
\|\Gamma_q[Q]\|^2  =\| (\Gamma_q[Q])^*  \Gamma_q[Q]\|  = \| \Gamma_q[|Q|^2]\| = \| |Q|^2\|_\infty = \| Q\|_\infty^2.
\]
The desired equality \eqref{Q-norm-eq} now follows immediately.
\end{proof}

\begin{remark}
Proposition \ref{prop-norm-poly} can be derived from the von-Neumann inequality. Indeed,
Theorem \ref{thm-gen-mul} and Corollary \ref{cor-isometry} imply that
\begin{itemize}
\item $\Gamma_q[e^{i\theta}]$ is an isometry on $\ell^2(T_q)$ and in particular $\| \Gamma_q[e^{i\theta}]\|\le 1$.
\item  $(\Gamma_q[e^{i\theta}])^n = \Gamma_q[e^{in\theta}]$ for any $n\in \N$ and thus $\Gamma_q[Q]   = Q(\Gamma_q[e^{i\theta}])$ for any $Q\in \C[e^{i\theta}]$.
\end{itemize}
Then by the von Neumann inequality (see, e.g., Pisier \cite[Chapter 1]{Pisier-similarity}), we obtain
\[
\| \Gamma_q[Q]\| \le \| Q\|_\infty = \sup_{\theta \in \R} | Q(e^{i\theta})|.
\]
\end{remark}

\begin{proposition}\label{prop-tri-symbol}
For any  trigonometric polynomial $P \in \C[e^{i\theta}, e^{-i \theta}]$, we have
\begin{align}\label{P-norm-eq}
\|\Gamma_q[P]\| = \| P\|_\infty.
\end{align}
\end{proposition}

\begin{proof}
Fix any trigonometric polynomial $P \in \C[e^{i\theta}, e^{-i \theta}]$. Write
\[
P = \sum_{n = -N}^N a_n e^{i n \theta} = e^{-i N \theta} \cdot \underbrace{\sum_{n = -N}^N a_n e^{i (n + N)\theta}}_{\text{denoted  by $Q$}}.
\]
Then $Q  = e^{iN \theta } P \in \C[e^{i\theta}]$ and hence by Lemma \ref{lem-mul}, we have
\[
\Gamma_q[P] = \Gamma_q[ e^{-i N\theta} Q] = \Gamma_q[e^{- i N \theta}] \Gamma_q[Q] =  (\Gamma_q[e^{ i N \theta}])^* \Gamma_q[Q].
\]
Therefore, by Proposition \ref{prop-norm-poly},
\begin{align*}
\| \Gamma_q[P]\| &   =  \| (\Gamma_q[e^{ i N \theta}])^* \Gamma_q[Q]\| \le  \| \Gamma_q[e^{ i N \theta}]\| \cdot \|  \Gamma_q[Q]\|
\\
& = \| e^{i N \theta}\|_\infty \cdot \| Q\|_\infty = \| e^{i N \theta}\|_\infty \cdot \| e^{i N \theta} P\|_\infty = \| P \|_\infty.
\end{align*}
The equality  \eqref{P-norm-eq} now follows immediately by applying the inequality \eqref{sup-norm}.
\end{proof}

We are ready to prove Theorem \ref{thm-all-symbol}. Recall that the $N$-th F\'ejer kernel  defined by
\begin{align}\label{def-Fejer-kernel}
\mathcal{F}_N  (e^{i\theta}): = \sum_{| k | \le N} \left(1 - \frac{|k|}{N+1}\right)  \cdot e^{ik \theta}.
\end{align}
Recall also that for any $f \in L^\infty(\T)$, the convolution $\mathcal{F}_N* f \in \C[e^{i\theta}, e^{-i \theta}]$ and we have
\begin{align}\label{Fejer-contract}
\|\mathcal{F}_N* f\|_\infty \le \| f \|_\infty.
\end{align}
\begin{proof}[Proof of Theorem \ref{thm-all-symbol}]
The set-theoretical equality \eqref{set-th-eq} has already been proved in Corollary \ref{cor-set-eq}.  By Lemma \ref{lem-one-side}, for proving the equality \eqref{norm-2-side}, it remains to prove that for any $f \in L^\infty(\T)$, we have $\| \Gamma_q[f]\|\le \| f\|_\infty$.
Recall the definition \eqref{trunc-toep} of the truncated operator $\Gamma_q^{(n)}[f]$.  Clearly, we have
\[
\|\Gamma_q[f]\| = \sup_{n \ge 1} \| \Gamma_q^{(n)}[f]\|.
\]
Therefore, we only need to show that for any $n \ge 1$, we have
\begin{align}\label{norm-finite-matrix}
\| \Gamma_q^{(n)}[f]\|\le \| f \|_\infty.
\end{align}
But for any fixed integer $n \ge 1$, since the coefficients of the finite matrix $\Gamma_q^{(n)}[f]$ involves only finitely many coefficients $(\widehat{f}(k))_{|k|\le n}$, we have
\[
\lim_{N\to\infty}\| \Gamma_q^{(n)}[f] - \Gamma_q^{(n)}[\mathcal{F}_N* f]\| = 0.
\]
Since $\mathcal{F}_N*f \in \C[e^{i\theta}, e^{-i \theta}]$, we may  apply Proposition \ref{prop-tri-symbol} and use \eqref{Fejer-contract} to obtain
\begin{align*}
\|\Gamma_q^{(n)}[f] \| = \lim_{N\to\infty}\|   \Gamma_q^{(n)}[\mathcal{F}_N* f]\| \le \limsup_{N\to\infty} \| \mathcal{F}_N*f\|_\infty \le \|f \|_\infty.
\end{align*}
This completes the proof of the inequality \eqref{norm-finite-matrix} and hence completes the whole proof.
\end{proof}

\section{The method of full Fock space}\label{sec-fock}

\subsection{Proof of Theorem \ref{thm-direct-sum}}
For $q\ge 2$, let $\mathcal{F}(\C^q)$ denote the Hilbert space of  full Fock space of the Euclidean space $\C^q$:
\[
\mathcal{F}(\C^q): = \C \Omega \oplus \bigoplus_{n = 1}^\infty (\C^q)^{\otimes n},
\]
where $\C\Omega$ is the one-dimensional complex Euclidean space, $\Omega$ is the unit of $\C$  and $(\C^q)^{\otimes n}$ is the Hilbertian tensor product.  Let $e_1, \cdots, e_q$ be the natural basis of $\C^q$, then an orthonormal basis of $\mathcal{F}(\C^q)$ is given by
\begin{align}
\mathscr{B}: = \{ \Omega\} \cup  \bigcup_{n\ge 1} \Big\{e_{i_1} \otimes \cdots \otimes e_{i_n}\Big| (i_1, \cdots, i_n) \in \{1, \cdots, q\}^n \Big\}.
\end{align}
Recall that we denote by $\F_q^{+}$  the unital free semi-group generated by $q$ free elements $s_1, \cdots, s_q$ with   the neutral element $e\in \F_q^+$. Clearly, the natural bijection between $\mathscr{B}$ and $\F_q^+$ (that is, $e_{i_1}\otimes  \cdots \otimes e_{i_n} \mapsto s_{i_1} \cdots s_{i_n}$ and $\Omega \mapsto e$) induces a  unitary isomorphism
\begin{align}\label{def-W}
W:  \ell^2(\F_q^{+})  \xrightarrow{\quad \simeq\quad}  \mathcal{F}(\C^q).
\end{align}
Recall the bijection $\iota: \F_q^+ \rightarrow T_q$ fixed in \eqref{def-iota}. This bijection (combined with the unitary operator $W$ in \eqref{def-W}) induces a natural unitary isomorphism:
\[
U_\iota:  \ell^2(T_q) \xrightarrow{\quad \simeq \quad}\mathscr{F}(\C^q).
\]
For any $h \in \C^q$, we define the left creation operator $\ell(h): \mathcal{F}(\C^q) \rightarrow \mathcal{F}(\C^q)$  by setting
\[
\ell(h) v = h \otimes v, \quad \forall v \in \mathcal{F}(\C^q).
\]

\begin{lemma}\label{lem-gamma-1}
We have the following commutative diagram:
\begin{align}\label{comm-diag-fock}
\begin{CD}
\ell^2(T_q) @>  \quad \Gamma_q[e^{i\theta}] \quad>>   \ell^2(T_q)
\\
@V{U_\iota}V{\simeq }V     @V{U_\iota}V{\simeq}V
\\
 \mathcal{F}(\C^q) @>\, \ell(\frac{e_1 + \cdots + e_q}{\sqrt{q}}) \, >>  \mathcal{F}(\C^q)
\end{CD}.
\end{align}
\end{lemma}

\begin{proof}
It suffices to show the following equality
\begin{align}\label{gamma-1-form}
\Gamma_q[e^{i\theta}] \delta_{\iota(w)}=  \frac{1}{\sqrt{q}} \sum_{k=1}^q   \delta_{\iota(s_k w)}.
\end{align}
In fact, by the definition of $\Gamma_q[e^{i\theta}]$,  for any $\sigma_1, \sigma_2 \in T_q$, we have
\begin{align*}
(\Gamma_q[e^{i\theta}] \delta_{\sigma_2})(\sigma_1)& = \sum_{\tau \in T_q} \Gamma_q[e^{i\theta}](\sigma_1, \tau) \delta_{\sigma_2}(\tau) =  \Gamma_q[e^{i\theta}] (\sigma_1, \sigma_2)
\\
& =  q^{- d(\sigma_1, \sigma_2)/2} \mathds{1}(| \sigma_1|- |\sigma_2| =1) \mathds{1}_{\mathcal{C}}(\sigma_1, \sigma_2)
\\
&  = \frac{1}{\sqrt{q}} \mathds{1}(\text{$\sigma_1$ is a child of $\sigma_2$}).
 \end{align*}
For any $\sigma \in T_q$,  let $D(\sigma)\subset T_q$ denote the set consists of all children of $\sigma$. Then
\[
\Gamma_q[e^{i\theta}] \delta_{\sigma} = \frac{1}{\sqrt{q}}\sum_{\tau \in D(\sigma)} \delta_\tau.
\]
This is exactly the desired equality \eqref{gamma-1-form}.
\end{proof}

Given any contractive operator $A: \C^q \rightarrow \C^q$, define $\mathcal{F}(A): \mathcal{F}(\C^q) \rightarrow \mathcal{F}(\C^q)$  by
\begin{align}\label{def-F-V}
\mathcal{F}(A) = Id_{\C \Omega} \oplus \bigoplus_{n\ge 1} A^{\otimes n}.
\end{align}
Clearly, if $A: \C^q \rightarrow \C^q$ is unitary, then so is $\mathcal{F}(A)$.

\begin{lemma}\label{lem-2-cr-op}
Let $V: \C^q \rightarrow \C^q$ be any  unitary operator (not unique) such that
\[
V \left(\frac{e_1+ \cdots + e_q}{\sqrt{q}} \right)=e_1.
\]
Then we have the following commutative diagram
\begin{align}\label{comm-2-creation}
\begin{CD}
\mathcal{F}(\C^q) @>\, \ell(\frac{e_1 + \cdots + e_q}{\sqrt{q}}) \, >>  \mathcal{F}(\C^q)
\\
@V{\mathcal{F}(V)}V{\simeq }V     @V{\mathcal{F}(V)}V{\simeq}V
\\
 \mathcal{F}(\C^q) @>\, \ell(e_1) \, >>  \mathcal{F}(\C^q)
\end{CD}.
\end{align}
\end{lemma}

\begin{proof}
This is immediate from the definition of $\mathcal{F}(V)$.
\end{proof}

\begin{lemma}\label{lem-cre-direct}
Let $S_1: \ell^2(\N)\rightarrow \ell^2(\N)$ be the left shift operator.
Then $\ell(e_1)$ is unitarily equivalent to the direct sum of the countably infinitely many $S_1$:
\[
\ell(e_1) \simeq
\left[
\begin{array}{cccc}
S_1 & & &
\\
& S_1 & &
\\
& & S_1 &
\\
& & &  \ddots
\end{array}
 \right].
\]
\end{lemma}

\begin{proof}
For any $n\in\N$, set
\[
e_1^{\otimes n} =
\left\{ \begin{array}{cl} \underbrace{e_1 \otimes \cdots \otimes e_1}_{\text{$n$-times}} & \text{if $n\ge 1$}
\vspace{2mm}
\\
\Omega & \text{if $n =0$}
\end{array}
\right..
\]
Define a closed subspace $\mathcal{H}\subset \mathcal{F}(\C^q)$ by
\[
\mathcal{H}: = \bigoplus_{n \ge 1}  \bigoplus_{(i_1, \cdots, i_n)\in \{1, \cdots, q\}^n \atop i_1 \ne 1 } \C e_{i_1}\otimes \cdots \otimes e_{i_n}
\]
Clearly, we have
\[
\mathcal{F}(\C^q) = \bigoplus_{n \in \N} \Big(\C e_1^{\otimes n} \otimes \mathcal{H} \Big) \simeq \ell^2(\N; \mathcal{H}).
\]
Now it is easy to see that $\ell(e_1)$ is unitarily equivalent to the left shift operator on the Hilbert space $\ell^2(\N; \mathcal{H}) = \ell^2(\N) \otimes \mathcal{H}$ and this completes the proof.
\end{proof}

\begin{proof}[Proof of Theorem \ref{thm-direct-sum}]
Clearly, we have $S_1 = \Gamma_1[e^{i\theta}] = T(e^{i\theta})$. By Lemma \ref{lem-gamma-1},  Lemma \ref{lem-2-cr-op} and  Lemma \ref{lem-cre-direct}, there exists a unitary operator $U: \ell^2(T_q) \rightarrow \bigoplus_{n\ge 1} \ell^2(\N)$ such that
\[
\Gamma_q[e^{i\theta}] = U^{-1} \left[
\begin{array}{cccc}
S_1 & & &
\\
& S_1 & &
\\
& & S_1 &
\\
& & &  \ddots
\end{array}
 \right]   U
=
U^{-1} \left[
\begin{array}{cccc}
\Gamma_1[e^{i\theta}] & & &
\\
& \Gamma_1[e^{i\theta}] & &
\\
& & \Gamma_1[e^{i\theta}] &
\\
& & &  \ddots
\end{array}
 \right]  U.
\]
Thus by Lemma \ref{lem-mul}, for any $Q\in \C[e^{i\theta}]$, we have
\[
\begin{split}
\Gamma_q[Q] & = Q(\Gamma_q[e^{i\theta}])  =U^{-1} \left[
\begin{array}{cccc}
Q(\Gamma_1[e^{i\theta}]) & & &
\\
&  Q(\Gamma_1[e^{i\theta}]) & &
\\
& &  Q(\Gamma_1[e^{i\theta}]) &
\\
& & &  \ddots
\end{array}
 \right]  U
\\
& = U^{-1} \left[
\begin{array}{cccc}
\Gamma_1[Q] & & &
\\
& \Gamma_1[Q] & &
\\
& & \Gamma_1[Q] &
\\
& & &  \ddots
\end{array}
 \right]  U.
\end{split}
\]
Then applying the equality $(\Gamma_q[f])^* = \Gamma_q[\bar{f}]$ and noting that  any $P \in \C[e^{i\theta}, e^{-i\theta}]$ can be written as $P = Q_1 + \bar{Q}_2$ with $Q_1, Q_2 \in \C[e^{i\theta}]$, we obtain
\begin{align}\label{trigo-dir-sum}
\begin{split}
\Gamma_q[P]  = U^{-1} \left[
\begin{array}{cccc}
\Gamma_1[P] & & &
\\
& \Gamma_1[P] & &
\\
& & \Gamma_1[P] &
\\
& & &  \ddots
\end{array}
 \right]  U.
\end{split}
\end{align}
It remains to extend the equality \eqref{trigo-dir-sum} to all $f\in L^\infty(\T)$. For this purpose, we now show that the map
\begin{align}\label{des-w-cont}
L^\infty(\T) \ni f \mapsto  U^{-1} \left[
\begin{array}{cccc}
\Gamma_1[f] & & &
\\
& \Gamma_1[f] & &
\\
& & \Gamma_1[f] &
\\
& & &  \ddots
\end{array}
 \right]  U  \in B(\ell^2(T_q))
\end{align}
is continuous with respect to the weak-star topology on $L^\infty(\T)$ and the weak operator topology on $B(\ell^2(T_q))$. In fact,  by recalling the commutative diagram \eqref{2-way-Toep} and  the elementary fact that the map $ B(\mathcal{H}) \ni B\mapsto ABC \in B(\mathcal{H})$ is continuous under the weak operator topology on $B(\mathcal{H})$ for any fixed $A, C\in B(\mathcal{H})$, it suffices to show the map
\begin{align}\label{des-w-cont-bis}
L^\infty(\T) \ni f \mapsto  DT(f): =   \left[
\begin{array}{cccc}
T_f & & &
\\
& T_f & &
\\
& & T_f &
\\
& & &  \ddots
\end{array}
 \right]  \in B\Big( \bigoplus_{n\ge 1} H^2(\T)\Big),
\end{align}
is continuous with respect to the weak-star topology on $L^\infty(\T)$ and the weak operator topology on $B\Big( \bigoplus_{n\ge 1} H^2(\T)\Big)$,  where $T_f$ is the Toeplitz operator defined in \eqref{Toep-hardy} on the Hardy space $H^2(\T)$. Indeed, take any two elements $\Psi = (\varphi_n)_{n\ge 1}$ and  $\Phi = (\psi_n)_{n\ge 1}$ in $\bigoplus_{n\ge 1}H^2(\T)$,  we have
\begin{align*}
\langle DT(f) \Psi, \Phi\rangle  & = \sum_{n=1}^\infty \langle P_{+} (f \varphi_n), \psi_n\rangle = \sum_{n= 1}^\infty \langle  f\varphi_n, \psi_n\rangle  = \sum_{n=1}^\infty \int_\T f \varphi_n \bar{\psi}_n.
\end{align*}
Now since
\[
\sum_{n=1}^\infty \| \varphi_n \bar{\psi}_n\|_1 \le \sum_{n=1}^\infty \| \varphi_n\|_2 \| \psi_n\|_2 \le   \Big(\sum_{n=1}^\infty \| \varphi_n\|_2^2\Big)^{1/2}  \Big(\sum_{n=1}^\infty \| \psi_n\|_2^2\Big)^{1/2}  = \| \Phi\| \| \Phi\|<\infty,
\]
the series
$
g:  = \sum_{n=1}^\infty \varphi_n \bar{\psi}_n
$
converges absolutely in $L^1(\T)$ and hence the map
\[
L^\infty(\T) \ni f \mapsto \langle DT(f) \Phi, \Psi\rangle  = \int_\T f g \in \C
\]
is continuous with respect to the weak-star topology on $L^\infty(\T)$. Since $\Phi, \Psi$ are chosen arbitrarily, we obtained the desired continuity of the map \eqref{des-w-cont-bis}.

Finally, the mentioned continuity of the map \eqref{des-w-cont} combined with the continuity of the map \eqref{weak-star-one-f} enables us to extend the equality \eqref{trigo-dir-sum}  to all  $f\in L^\infty(\T)$.
\end{proof}

\subsection{Proof of Theorem \ref{thm-generalization}}

Recall the unitary operator $W$ introduced in \eqref{def-W}.
\begin{lemma}\label{lem-gamma-a}
For any $a\in \BS(\C^q)\subset \C^q$, the kernel $\Gamma_{a}[e^{i\theta}]$ induces a bounded operator on $\ell^2(\F_q^{+})$ and  we have the following commutative diagram:
\begin{align}\label{comm-diag-fock-a}
\begin{CD}
\ell^2(\F_q^{+}) @>  \quad \Gamma_{a}[e^{i\theta}] \quad>>   \ell^2(\F_q^{+})
\\
@V{W}V{\simeq }V     @V{W}V{\simeq}V
\\
 \mathcal{F}(\C^q) @>\, \ell(a) \, >>  \mathcal{F}(\C^q)
\end{CD}.
\end{align}
\end{lemma}

\begin{proof}
It suffices to show the following equality holds for any word $w\in \F_q^{+}$:
\begin{align}\label{gamma-1-form-a}
\Gamma_{a}[e^{i\theta}] \delta_{w}=  \sum_{k=1}^q  a_k  \delta_{s_k w}.
\end{align}
In fact, by the definition of $\Gamma_{a}[e^{i\theta}]$, we have  for any $w_1, w_2 \in \F_q^{+}$,
\begin{align*}
(\Gamma_{a}[e^{i\theta}] \delta_{w_2})(w_1)  =  \Gamma_{a}[e^{i\theta}] (w_1, w_2) = \sum_{k=1}^q \mathds{1}(w_1 = s_k w_2) \cdot [s_k](a) = \sum_{k=1}^q a_k \delta_{s_k w_2} (w_1).
 \end{align*}
This gives exactly the desired equality \eqref{gamma-1-form-a}.
\end{proof}

\begin{lemma}\label{lem-unit-shift}
Let $V_a: \C^q \rightarrow \C^q$ be any  unitary operator  such that $V_a (a)=e_1$.
Then we have the following commutative diagram
\[
\begin{CD}
\mathcal{F}(\C^q) @>\, \ell(a) \, >>  \mathcal{F}(\C^q)
\\
@V{\mathcal{F}(V_a)}V{\simeq }V     @V{\mathcal{F}(V_a)}V{\simeq}V
\\
 \mathcal{F}(\C^q) @>\, \ell(e_1) \, >>  \mathcal{F}(\C^q)
\end{CD}.
\]
\end{lemma}

\begin{proof}
This is immediate from the definition \eqref{def-F-V} of $\mathcal{F}(V_a)$.
\end{proof}

\begin{lemma}\label{lem-conjugate}
Let $a\in \BS(\C^q)$. For any formal Fourier series $f$, we have
\[
\overline{\Gamma_{a}[f] (v, u)} = \Gamma_{a}[\bar{f}](u, v), \quad \forall u, v \in \F_q^{+},
\]
where $\bar{f}$ is the formal Fourier series defined by
\[
\bar{f} \sim \sum_{n\in \Z} \overline{\widehat{f}(n)} e^{- in \theta} = \sum_{n\in \Z} \overline{\widehat{f}(-n)} e^{in \theta}.
\]
\end{lemma}
\begin{proof}
This follows immediately from the definition \eqref{def-gamma-a-f}.
\end{proof}

For any formal Fourier series $f$, define the formal product kernel $\Gamma_{a}[f] \Gamma_{a}[e^{i\theta}]$ by
\begin{align}\label{def-prod}
(\Gamma_{a}[f] \Gamma_{a}[e^{i\theta}]) (u, v): = \sum_{w \in \F_q^{+}} \Gamma_{a} [f] (u, w) \Gamma_{a}[e^{i\theta}](w, v).
\end{align}
Note that \eqref{def-prod} is well-defined since for any $v\in \F_q^{+}$, there is only finitely many $w\in \F_q^{+}$ such that $\Gamma_{a}[e^{i\theta}](w, v) \ne 0$.  Similarly, for any $P \in \C[e^{i\theta}, e^{-i\theta}]$, the formal product kernels
\[
\Gamma_{a}[f] \Gamma_{a}[P], \quad \Gamma_{a} [P] \Gamma_{a}[f]
\]
are well-defined.

\begin{lemma}\label{lem-formal-prod}
Let $a\in \BS(\C^q)$. For any formal Fourier series $f$, we have
\[
\Gamma_{a}[e^{i\theta}f]  = \Gamma_{a}[f] \Gamma_{a}[e^{i\theta}],
\]
where $e^{i\theta}f$ is the formal Fourier series defined by
\[
e^{i\theta} f \sim \sum_{n\in \Z} \widehat{f}(n) e^{i (n+1) \theta}=  \sum_{n\in \Z} \widehat{f}(n-1) e^{i n\theta}.
\]
\end{lemma}

\begin{proof}
For brevity, set $g = e^{i\theta}f$.  By definition \eqref{def-prod}, for any $u, v \in \F_q^{+}$,
\begin{align}\label{2-gamma-uv}
(\Gamma_{a}[f] \Gamma_{a}[e^{i\theta}]) (u, v) = \sum_{k = 1}^q \Gamma_{a}[f](u, s_k v) a_k.
\end{align}
{\flushleft Case (i): $(u, v)\notin \mathcal{C}$.} In this case, we have $(u, s_k v)\notin \mathcal{C}$ for any $1\le k \le q$. It follows that $\Gamma_{a}[f] \Gamma_{a}[e^{i\theta}] (u, v) =0$. Hence the product kernel $\Gamma_{a}[f] \Gamma_{a}[e^{i\theta}]$ coincides with the kernel $\Gamma_{a}[g]$ on the complement of $\mathcal{C}$ (both of them vanish on the complement of $\mathcal{C}$).
{\flushleft Case (ii): $u\pless v$.} In this case, there exists $w\in \F_q^{+}$ such that $v = wu$. Then
\begin{align*}
\Gamma_{a}[f] (u, s_k v) & = \widehat{f}(|u| - |s_k v|)   \cdot  \overline{[s_kw](a)}
\\
& = \widehat{f}(|u| - |v|-1)   \cdot  \overline{ a_k \cdot [w](a)}
\\
& = \bar{a}_k \widehat{g} (|u| - |v|)   \cdot \overline{[w](a)}.
 \end{align*}
Therefore, by recalling the assumption $\sum_{k=1}^q |a_k|^2 =1$ and by \eqref{2-gamma-uv}, we obtain
\[
(\Gamma_{a}[f] \Gamma_{a}[e^{i\theta}]) (u, v)   = \sum_{k=1}^q |a_k|^2  \widehat{g} (|u| - |v|)    \overline{[w](a)} = \widehat{g} (|u| - |v|) \cdot \overline{[w](a)}.
\]
On the other hand, in case (ii), $\Gamma_{a}[g] (u, v)  = \widehat{g} (|u| - |v|) \cdot \overline{[w](a)}$ and we have
\[
\Gamma_{a}[f] \Gamma_{a}[e^{i\theta}] (u, v) =  \Gamma_{a}[g](u, v).
\]
{\flushleft Case (iii): $v\pless u$ and $u\ne v$.}  In this case, there exist $l \in \{1, \cdots, q\}, w \in \F_q^{+}$ such that $u= ws_l v$. Then
\begin{align*}
\Gamma_{a}[f] (u, s_k v)  &  = \Gamma_{a}[f](w s_l v, s_k v)
\\
& = \mathds{1}(k=l) \cdot \widehat{f}(|ws_l v| - |s_k v|) \cdot  [w] (a)
\\
& = \mathds{1}(k = l)\cdot \widehat{f}(|u| - |v|-1)\cdot [w](a)
\\
& = \mathds{1}(k = l)\cdot \widehat{g}(|u| - |v|)\cdot [w](a).
\end{align*}
Hence by \eqref{2-gamma-uv}, we obtain
\[
(\Gamma_{a}[f] \Gamma_{a}[e^{i\theta}]) (u, v)  =\widehat{g}(|u|-|v|)\cdot [w](a) a_l = \widehat{g}(|u|-|v|) \cdot [ws_l](a).
\]
On the other hand, in case (iii), $\Gamma_{a}[g] (u, v)  = \widehat{g} (|u| - |v|) \cdot [ws_l](a)$ and we have
\[
\Gamma_{a}[f] \Gamma_{a}[e^{i\theta}] (u, v) =  \Gamma_{a}[g](u, v).
\]
Combining the above three cases, we complete the proof of the lemma.
\end{proof}

\begin{corollary}\label{cor-mul}
Let $a\in \BS(\C^q)$. For any formal Fourier series $f$ and any $Q\in \C[e^{i\theta}]$, 
\begin{align}\label{f-Q-prod}
\Gamma_{a}[f]\Gamma_{a}[Q] = \Gamma_{a}[fQ],
\end{align}
where $fQ = Qf$ is the formal Fourier series defined by the product of the trigonometric polynomial $Q$ and the formal Fourier series $f$.
\end{corollary}
\begin{proof}
By Lemma \ref{lem-formal-prod} and an induction argument,
we obtain $\Gamma_{a}[f]\Gamma_{a}[e^{i n\theta}] = \Gamma_{a}[e^{in\theta} f]$ for all $n\in \N$.  We  complete the proof by linearity  in $Q$ of both sides of the equality  \eqref{f-Q-prod}.
\end{proof}

\begin{lemma}\label{lem-trigo-general}
Let $a\in \BS(\C^q)$. For any  $Q\in \C[e^{i\theta}]$, the kernel $\Gamma_{a}[Q]$ induces a bounded operator  on $\ell^2(\F_q^{+})$ and  we have the following equality of bounded operators:
\begin{align}\label{Q-gamma-a}
\Gamma_{a}[Q] = Q( \Gamma_{a}[e^{i\theta}]).
\end{align}
Consequently, for any $P\in \C[e^{i\theta}, e^{-i \theta}]$, the kernel $\Gamma_a[P]$ induces a bounded operator. Moreover, there exists a unitary operator $U_a: \ell^2(\F_q^{+}) \rightarrow  \bigoplus_{k=1}^\infty \ell^2(\N)$ such that for any $P \in \C[e^{i\theta}, e^{-i\theta}]$, we have
\begin{align}\label{trigo-dir-sum-general}
\begin{split}
\Gamma_a[P]  = U_a^{-1} \left[
\begin{array}{cccc}
\Gamma_1[P] & & &
\\
& \Gamma_1[P] & &
\\
& & \Gamma_1[P] &
\\
& & &  \ddots
\end{array}
 \right]  U_a.
\end{split}
\end{align}
\end{lemma}

\begin{proof}
Lemma \ref{lem-gamma-a} implies that $\Gamma_a[e^{i\theta}]$ is a bounded operator. Therefore, by Corollary \ref{cor-mul}, we have  the equality of kernels $\Gamma_a[e^{in\theta}] = \Gamma_a[e^{i\theta}] \cdots \Gamma_a[e^{i\theta}]$ for all $n\in \N$ (product is understood as formal product of kernels). Thus $\Gamma_a[e^{in\theta}]$ is a bounded  operator for all $n\in \N$. It follows that $\Gamma_a[Q]$ is a bounded operator for all $Q\in \C[e^{i\theta}]$. In particular, the equalities  $\Gamma_a[e^{in\theta}] = \Gamma_a[e^{i\theta}] \cdots \Gamma_a[e^{i\theta}]  = (\Gamma_a[e^{i\theta}])^n$ (now understood as equalities of bounded operators), combined with the linearity,  imply the desired equality \eqref{Q-gamma-a}.  The second assertion is proved as follows:  if $P = Q_1 + \bar{Q}_2$ with $Q_1, Q_2 \in \C[e^{i\theta}]$,  then by applying Lemma \ref{lem-conjugate}, we have  $\Gamma_a[P] = \Gamma_a[Q_1] + (\Gamma_a[Q_2])^*$.

Finally, the proof of the equality \eqref{trigo-dir-sum-general} is similar to that of the equality \eqref{trigo-dir-sum} and  follows from Lemmas \ref{lem-cre-direct}, \ref{lem-gamma-a}, \ref{lem-unit-shift} and the equality \eqref{Q-gamma-a}.
\end{proof}

\begin{proof}[Proof of Theorem \ref{thm-generalization}]
We first extend the equality \eqref{trigo-dir-sum-general} to all $f\in L^\infty(\T)$. For any $f\in L^\infty(\T)$, we define temporarily a bounded operator $\widetilde{\Gamma}_a[f]$ as follows:
\[
\widetilde{\Gamma}_a[f]:  = U_a^{-1} \left[
\begin{array}{cccc}
\Gamma_1[f] & & &
\\
& \Gamma_1[f] & &
\\
& & \Gamma_1[f] &
\\
& & &  \ddots
\end{array}
 \right]  U_a.
\]

Fix any $u, v \in \F_q^{+}$. The equality \eqref{trigo-dir-sum-general} means that $\Gamma_a[P] = \widetilde{\Gamma}_a[P]$ for all $P\in \C[e^{i\theta}, e^{-i\theta}]$ and in particular, we have
\begin{align}\label{gamma-tilde-gamma}
\Gamma_a[P] (u, v)= \widetilde{\Gamma}_a[P] (u, v).
\end{align}
Clearly,  the map $L^\infty(\T)\ni f\mapsto \Gamma_a[f](u, v) \in \C$ is continuous with respect to the weak-star topology on $L^\infty(\T)$. On the other hand, similar to the proof of the continuity result for the map \eqref{des-w-cont}, we can show that the map $L^\infty(\T) \ni f \mapsto \widetilde{\Gamma}_a[f]$ is continuous with respect to the weak-star topology on $L^\infty(\T)$ and the weak operator topology on $B(\ell^2(\F_q^{+}))$. In particular, it follows that the map $L^\infty(\T)\ni f\mapsto  \widetilde{\Gamma}_a[f](u, v) \in \C$ is continuous with respect to the weak-star topology on $L^\infty(\T)$. Therefore, we may extend the equality \eqref{gamma-tilde-gamma} to all $f\in L^\infty(\T)$:
\[
\Gamma_a[f] (u, v)= \widetilde{\Gamma}_a[f] (u, v).
\]
Since $u, v$ are chosen arbitrarily, we obtain the equality of two kernels $\Gamma_a[f] = \widetilde{\Gamma}_a[f]$. Consequently, the kernel $\Gamma_a[f]$ induces a bounded operator (since so does $\widetilde{\Gamma}_a[f]$) and the desired equality \eqref{gamma-direct-gamma}  holds.

It remains to show that if the kernel $\Gamma_{a}[f]$ defines a bounded operator on $\ell^2(\F_q^{+})$, then  $f\in L^\infty(\T)$. Note that if the kernel $\Gamma_{a}[f]$ defines a bounded operator, then so is the new kernel  $\Gamma_{a, t}[f]$ obtained by  gauge transformation of $\Gamma_a[f]$ for any $t\in \R/2\pi \Z$:
\begin{align}\label{def-gauge-trans}
\Gamma_{a, t}[f](u, v):  = e^{-i t |u|} \Gamma_a[f](u,v) e^{i t|v|}.
\end{align}
Moreover, we have the operator norm equality
\begin{align}\label{2-norm-eq}
\| \Gamma_{a, t}[f]\| = \| \Gamma_a[f]\|.
\end{align}
For any $t\in \R/2\pi \Z$,  set a new formal Fourier transform $f_t$ by
\begin{align}\label{def-f-t}
f_t \sim \sum_{n\in \Z} \widehat{f}(n) e^{-i nt} e^{i n \theta}.
\end{align}
Clearly, from the definition \eqref{def-gamma-a-f} and  the definition \eqref{def-gauge-trans}, we have
\begin{align}\label{id-gauge-trans}
\Gamma_{a, t}[f]  = \Gamma_a[f_t].
\end{align}
Recall the notaion $\mathcal{F}_N$ in \eqref{def-Fejer-kernel} for the F\'ejer kernel. Note that we have
\begin{align}\label{def-F-N-f}
\int_\T f_t \mathcal{F}_N(e^{it}) \frac{dt}{2\pi} =  \underbrace{ \sum_{|k|\le N} \left(1 - \frac{|k|}{N+1}\right) \widehat{f}(n) e^{in\theta}}_{=: \mathcal{F}_N* f} \in \C[e^{i\theta}, e^{-i\theta}].
\end{align}
Since the map $f \mapsto \Gamma_q[f]$ is linear from the linear  space of formal Fourier series to the set of kernels on $\F_q^{+}$, for any $N\ge 1$,  we have
\[
\Gamma_a[\mathcal{F}_N*f]  =   \Gamma_{a}\left[\int_\T f_t \mathcal{F}_N(e^{it}) \frac{dt}{2\pi}  \right] = \int_\T  \Gamma_{a}[f_t] \mathcal{F}_N(e^{it}) \frac{dt}{2\pi} = \int_\T  \Gamma_{a, t}[f] \mathcal{F}_N(e^{it}) \frac{dt}{2\pi}.
\]
Therefore, by noting the equality \eqref{2-norm-eq} and elementary inequality $\mathcal{F}_N(e^{it}) \ge 0$, we have
\[
\| \Gamma_a[\mathcal{F}_N*f]\| \le \int_\T \| \Gamma_{a,t}[f] \mathcal{F}_N(e^{it})\| \frac{dt}{2\pi} =  \int_\T \| \Gamma_{a}[f] \| \mathcal{F}_N(e^{it}) \frac{dt}{2\pi}  = \| \Gamma_{a}[f] \|.
\]
But since $\mathcal{F}_N*f\in \C[e^{i\theta}, e^{-i\theta}]$ for any integer $N\ge 1$, we can  use the equality \eqref{trigo-dir-sum-general} and the classical result on standard Toeplitz operators to conclude that
\[
\| \Gamma_a[\mathcal{F}_N*f]\| = \| \mathcal{F}_N*f\|_{L^\infty(\T)}.
\]
 Thus we obtain
\[
\| \mathcal{F}_N*f\|_{L^\infty(\T)} \le \| \Gamma_a[f]\|, \quad \forall N\ge 1.
\]
Then by a standard argument, we obtain $f\in L^\infty(\T)$.
\end{proof}

\subsection{A new proof of Theorem \ref{thm-QW-I}}\label{sec-new-pf}

Before proceeding to the new proof of Theorem \ref{thm-QW-I}, let us give the following warning.
{\flushleft \bf Warning:} Thereom \ref{thm-QW-I} was used in our  proof of the implication $f\in L^\infty(\T)\Longrightarrow \Gamma_q[f]\in B(\ell^2(T_q))$. Therefore, in  the following new proof of Theorem \ref{thm-QW-I}, we should avoid the use of  the above implication.  However,  since the proof of the equality \eqref{trigo-dir-sum} does not involve Theorem \ref{thm-QW-I},  we are allowed to use the implication $P\in \C[e^{i\theta}, e^{-i\theta}] \Longrightarrow \Gamma_q[P] \in B(\ell^2(T_q)$.

\bigskip

We will use  the classical result on Toeplitz operators: let $f\in L^\infty(\T)$, then $T(f)=\Gamma_1[f]$ is a positive operator if and only if $f \ge 0$.

Recall the notaion $\mathcal{F}_N$ in \eqref{def-Fejer-kernel} for the F\'ejer kernel. Let $\mu$ be any positive Radon measure on $\T$, we want to show that the kernel $\Gamma_q[\mu]$ is positive definite. Indeed, for any $N\ge 1$,  the trigonometric polynomial $\mathcal{F}_N* \mu\in \C[e^{i\theta}, e^{-i\theta}]$ defines a non-negative continuous function on $\T$. Therefore, by equality \eqref{trigo-dir-sum} and the above classical result on standard Toeplitz operators, the kernel $\Gamma_q[\mathcal{F}_N*\mu]$ is positive definite. It follows that, as the coordinatewise limit of $\Gamma_q[\mathcal{F}_N*\mu]$ as $N\to\infty$, the kernel $\Gamma_q[\mu]$ is also positive definite.

Conversely, assume that for a formal Fourier series $f$, the kernel $\Gamma_q[f]$ is positive definite. We want to show that there exists a positive Radon measure $\mu$ on $\T$ such that $\Gamma_q[f] = \Gamma_q[\mu]$.  Indeed, since the kernel $\Gamma_q[f]$ is positive definite, so is the following kernel $\Gamma_{q, t}[f]$ (for any $t \in \R/2\pi \Z$) after a gauge transformation
\[
\Gamma_{q, t}[f]( \sigma_1, \sigma_2) : =  e^{-i t |\sigma_1|}\Gamma_q[f](\sigma_1, \sigma_2) e^{i t |\sigma_2|} = q^{-d(\sigma_1, \sigma_2)/2} \widehat{f}(|\sigma_1| - |\sigma_2|)  e^{-i t (|\sigma_1|-|\sigma_2|)} \mathds{1}_{\mathcal{C}}(\sigma_1, \sigma_2).
\]
Recall the definition \eqref{def-f-t} for  $f_t$, similar to the equality \eqref{id-gauge-trans}, here we have
$
\Gamma_{q, t}[f] = \Gamma_q[f_t].
$
Recall the equality \eqref{def-F-N-f}. Since the map $f \mapsto \Gamma_q[f]$ is linear from the linear  space of formal Fourier series to the set of kernels on $T_q$ and $\mathcal{F}_N(e^{it}) \ge 0$ on $\T$, the positive definiteness of $\Gamma_q[f_t]$ for all $t\in \T$ implies the positive definiteness of the following  kernel:
\begin{align*}
\Gamma_q\left[\mathcal{F}_N* f\right]  = \Gamma_q\left[\int_\T f_t \mathcal{F}_N(e^{it}) \frac{dt}{2\pi}\right] = \int_\T \Gamma_q[f_t] \mathcal{F}_N(e^{it}) \frac{dt}{2\pi}.
\end{align*}
 By the equality \eqref{trigo-dir-sum}, the kernel  $\Gamma_q\left[\mathcal{F}_N* f\right]$ defines a bounded operator. Note that the positive definiteness of the kernel of a bounded operator is equivalent to the condition that the bounded operator is positive. Therefore,  the standard Toeplitz operator
\[
\Gamma_1\left[ \mathcal{F}_N* f \right]  = T\left( \mathcal{F}_N* f\right)
\]
is positive and hence   $\mathcal{F}_N* f \ge 0$  on $\T$. In particular,
\[
  \| \mathcal{F}_N* f\|_1= \int_\T  \mathcal{F}_N * f =  \widehat{f}(0).
\]
It follows that  $(\mathcal{F}_N*f)_{N=1}^\infty$ defines a sequence of positive Radon measures on $\T$ with the same total weight $\widehat{f}(0)$. Clearly,  for any $n\in \Z$,  the Fourier coefficients converges:
\[
\lim_{N\to\infty}\widehat{\mathcal{F}_N*f}(n)  = \widehat{f}(n).
\]
Therefore, we have the weak convergence of positive Radon measure:
\[
\lim_{N\to\infty} (\mathcal{F}_N*f)(e^{i\theta}) \frac{d\theta}{2\pi} = \mu(d\theta),
\]
where $\mu$ is a positive Radon measure on $\T$ with total weight $\mu(\T) = \widehat{f}(0)$. Hence the formal Fourier series coincides with the Fourier series of the positive Radon measure $\mu$.

The proof is complete.

\subsection{Proof of Theorem \ref{thm-gen-pos}} The proof of  Theorem \ref{thm-gen-pos} is  similar to our new proof of Theorem \ref{thm-QW-I} in \S \ref{sec-new-pf}.

\subsection{Proof of Theorem \ref{prop-pos-op}}

\begin{lemma}\label{lem-crea-tensor}
For any $A = (A_1, \cdots, A_q)\in B(\mathcal{H})^q$, we have the unitary equivalence
\begin{align}\label{uni-eq-gamma}
\Gamma_A[e^{i\theta}] \sim \ell(e_1)\otimes A_1 + \cdots + \ell(e_q) \otimes A_q.
\end{align}
In particular,
\begin{align}\label{n-eq-gamma}
\| \Gamma_A[e^{i\theta}]\| =  \Big\| \sum_{k=1}^q   A_k^* A_k\Big\|^{1/2}.
\end{align}
\end{lemma}
\begin{proof}
The unitary equivalence \eqref{uni-eq-gamma} follows from the equality
\[
\Gamma_A[e^{i\theta}] =   (W^{-1}\otimes Id)(\ell(e_1) \otimes A_1 + \cdots + \ell(e_q) \otimes A_q) (W \otimes Id),
\]
where $W$ is the unitary operator defined in \eqref{def-W} and $Id$ is the identity operator on $\mathcal{H}$. The norm equality \eqref{n-eq-gamma} follows from the equalities
\[
\ell(e_j)^* \ell(e_i) = \mathds{1}(i = j)  \cdot Id, \quad \forall i, j \in \{1, \cdots, q\}
\]
and
\[
(\Gamma_A[e^{i\theta}])^* \Gamma_A[e^{i\theta}]  = Id \otimes \Big(\sum_{k=1}^q A_k^* A_k\Big),
\]
where $Id$ stands for the identity operator on $\mathcal{F}(\C^q)$.
\end{proof}

\begin{proof}[Proof of Theorem \ref{prop-pos-op}]
By definition \eqref{def-gamma-a-f-op}, we have
\[
\Gamma_A[e^{i\theta}] (u, v)=  \left\{
\begin{array}{cc}
  [s_k](A) = A_k , &   \text{if $u = s_k v$}
\vspace{2mm}
\\
0, & \text{otherwise}
\end{array}
\right..
\]
Then it is easy to show that for any $n\ge 1$,
\begin{align}\label{def-gamma-A-n}
(\Gamma_A[e^{i\theta}])^n (u, v)=  \left\{
\begin{array}{cl}
 [w](A), &   \text{if $u = w v$ with $|w| = n$}
\vspace{2mm}
\\
0, & \text{otherwise}
\end{array}
\right..
\end{align}
It follows that
\begin{align}\label{def-gamma-A-adj}
(\Gamma_A[e^{i\theta}]^{*})^{n} (u, v)=  \left\{
\begin{array}{cl}
( [w](A))^*, &   \text{if $v = w u$ with $|w| = n$}
\vspace{2mm}
\\
0, & \text{otherwise}
\end{array}
\right..
\end{align}
Combining \eqref{def-gamma-A-n} and \eqref{def-gamma-A-adj} with \eqref{def-gamma-a-f-op}, for any $n\ge 1$, we have
\begin{align}\label{}
(\Gamma_A[e^{i\theta}])^n  = \Gamma_A[e^{i n \theta}],  \quad  (\Gamma_A[e^{i\theta}]^{*})^{n} =  \Gamma_A[e^{-in \theta}].
\end{align}
Then by von Neumann's inequality and \cite[Thm 2.6]{Paulsen-CB},  for any $P \in \C[e^{i\theta}, e^{-i\theta}]$, we have
\[
\| \Gamma_A[P]\| \le \| P\|_\infty.
\]
The above inequality clearly can be extended to all continuous function $f\in C(\T)$. We then complete the proof of the inequality $\| \Gamma_A[f]\| \le \| f\|_\infty$ for all $f\in L^\infty(\T)$ by applying the equality
\[
\| \Gamma_A[f]\| = \sup_{n \ge 1} \Big\| \Big[\Gamma_A[f](u, v)\Big]_{|u|, |v|\le n}\Big\|
\]
and the continuity of the map $L^\infty(\T)\ni f \mapsto \Gamma_A[f](u, v)$ for any fixed $u, v\in \F_q^{+}$ with respect to the weak-star topology on $L^\infty(\T)$ and norm topology on $B(\mathcal{H})$.  The last assertion on the complete positivity of the map $\Gamma_A$ follows from standard results in the theory of operator systems, see, e.g., \cite[Prop. 2.11 and Thm 3.11]{Paulsen-CB}.
\end{proof}

\subsection{Proof of Corollary \ref{cor-op-pos}}
Assume that $\mu$ is a positive Radon measure on $\T$. Then for any $N\ge 1$, the trigonometric polynomial $\mathcal{F}_N*\mu$ represents a positive function in $L^\infty(\T)$ and hence by Theorem \ref{prop-pos-op}, the operator-valued kernel $\Gamma_A[\mathcal{F}_N*\mu]$ is positive definite. Since the kernel $\Gamma_A[\mu]$ is the coordinatewise limit of $\Gamma_A[\mathcal{F}_N*\mu]$ as $N\to\infty$, it is also positive definite.

\section{Branching-Toeplitz matrices}

\subsection{Proof of Theorem \ref{thm-con-ext}: Case (A1)}
Fix an integer $n\ge 1$.  Assume that the formal Fourier series $f$ given in \eqref{fourier-series} satisfies $\widehat{f}(k)\ge 0$ for all integers $k$ with $|k|\le n$.
Then all the coefficients of $\Gamma_q^{(n)}[f]$ are non-negative. By the equality \eqref{norm-sup-2},  to complete the proof  of Theorem \ref{thm-con-ext} in case (A1),  we only need to prove
\[
 \left\| \Gamma_q^{(n)}[f]\Big|_{ \C^{\B_n(T_q)} \ominus \C^{\mathbb{B}_n(T_q)}_{\mathrm{rad}}}\right\| \le \left\| \Gamma_q^{(n)}[f]\Big|_{\C^{\mathbb{B}_n(T_q)}_{\mathrm{rad}}}\right\|.
\]
Assume by contradiction that
\begin{align}\label{assump-m-M}
\underbrace{\left\| \Gamma_q^{(n)}[f]\Big|_{\C^{\mathbb{B}_n(T_q)}_{\mathrm{rad}}}\right\|}_{\text{denoted by $m \ge 0$}} <\underbrace{ \left\| \Gamma_q^{(n)}[f]\Big|_{ \C^{\B_n(T_q)} \ominus \C^{\mathbb{B}_n(T_q)}_{\mathrm{rad}}}\right\|}_{\text{denoted by $M > 0$}}.
\end{align}
Then by the equality \eqref{norm-sup-2}, we have
\begin{align}\label{M-norm-gamma}
M = \| \Gamma_q^{(n)}[f]\|.
\end{align}
By a compactness argument, there exists $v = (v_\sigma)_{\sigma \in \B_n(T_q)} \in \C^{\B_n(T_q)} \ominus \C^{\mathbb{B}_n(T_q)}_{\mathrm{rad}}$ such that
\[
\| v\|=1 \an  M = \| \Gamma_q^{(n)} v\|.
\]
Define a vector $|v| \in \C^{\B_n(T_q)}$ by setting
\[
|v|_\sigma : = | v_\sigma|, \quad \forall \sigma \in \B_n(T_q).
\]

{\flushleft \bf Claim I:}  $|v|$ is a norming vector for $\Gamma_q^{(n)}[f]$: that is,
\[
\||v|\|=1 \an \| \Gamma_q^{(n)}[f] |v|\| = M = \| \Gamma_q^{(n)}[f]\|.
\]

Indeed, by definition of $|v|$, we clearly have $\| |v|\| = \| v\| = 1$. Moreover, since all coefficients of $\Gamma_q^{(n)}[f]$ are non-negative, we have
\[
 \| \Gamma_q^{(n)}[f] |v|\| \ge \| \Gamma_q^{(n)}[f] v\| = M = \| \Gamma_q^{(n)}[f]\|.
\]
We complete the proof of Claim I by applying the following inequality
\[
\| \Gamma_q^{(n)}[f] |v|\| \le  \| \Gamma_q^{(n)}[f]\| \cdot \| |v|\| = \| \Gamma_q^{(n)}[f]\|.
\]

Recall the notation  $P_\rad^{(n)}: \C^{\B_n(T_q)} \rightarrow \C^{\B_n(T_q)}_\rad$   of the orthogonal projection  onto $\C^{\B_n(T_q)}_\rad$. Denote also  the orthogonal projection
\[
P_\rad^{(n) \bot}: \C^{\B_n(T_q)} \xrightarrow{\quad} \C^{\B_n(T_q)} \ominus \C^{\mathbb{B}_n(T_q)}_{\mathrm{rad}}.
\]
Then we can write
\[
|v| = P_\rad^{(n)} (|v|) + P_\rad^{(n) \bot} (|v|).
\]
Clearly, we have
\begin{align}\label{dec-1}
1 = \| |v|\|^2 = \|  P_\rad^{(n)} (|v|)\|^2 + \| P_\rad^{(n)\bot} (|v|) \|^2.
\end{align}

{\flushleft \bf Claim II:} We have
\begin{align}\label{non-zero-proj}
P_\rad^{(n)} (|v|) \ne 0.
\end{align}
Otherwise, $P^{(n)}_{\mathrm{rad}}(|v|)=0$ and thus  $|v|\in \C^{\B_n(T_q)} \ominus \C^{\mathbb{B}_n(T_q)}_{\rad}$. But observe that a vector $w\in \C^{\B_n(T_q)} \ominus \C^{\mathbb{B}_n(T_q)}_{\rad}$ if and only if for any integer $k$ with $0 \le k\le n$, we have
\[
\sum_{\sigma:|\sigma|=k} w_\sigma=0.
\]
By definition, all the coefficients of the vector $|v|$ are non-negative. Therefore, if $|v|$ belongs to the space  $\C^{\B_n(T_q)} \ominus \C^{\mathbb{B}_n(T_q)}_{\rad}$, then $|v|$ and thus $v$ must be the zero-vector. This contradicts to the assumption $\|v\|_2=1$.

Now by using the block representation \eqref{BT-block-form-finite} of $\Gamma_q^{(n)}[f]$, we obtain the othogonal decomposition of $\Gamma_q^{(n)}[f] |v|$:
\begin{align}\label{rad-orth-dec}
\begin{split}
\Gamma_q^{(n)}[f] |v| =&   P_\rad^{(n)} \Gamma_q^{(n)}[f]  P_\rad^{(n)} (|v|)    +   P_\rad^{(n) \bot} \Gamma_q^{(n)}[f]  P_\rad^{(n) \bot} (|v|)
\\
= &  \big[\Gamma_q^{(n)}[f]\big|_{\C^{\B_n(T_q)}_\rad}\big] \big( P_\rad^{(n)}|v|\big) +  \big[\Gamma_q^{(n)}[f]\big|_{\C^{\B_n(T_q)} \ominus \C^{\B_n(T_q)}_\rad}\big] \big( P_\rad^{(n) \bot}|v|\big).
\end{split}
\end{align}
Then by Claim I, we have
\begin{align*}
M^2   \stackrel{\text{Claim I}}{=\joinrel=\joinrel=\joinrel=} &  \| \Gamma_q^{(n)}[f] |v|\|^2
\\
  \stackrel{\eqref{rad-orth-dec}}{=\joinrel=\joinrel=\joinrel=} &  \left\|  \big[\Gamma_q^{(n)}[f]\big|_{\C^{\B_n(T_q)}_\rad}\big] \big( P_\rad^{(n)}|v|\big) \right\|^2 + \left\| \big[\Gamma_q^{(n)}[f]\big|_{\C^{\B_n(T_q)} \ominus \C^{\B_n(T_q)}_\rad}\big] \big( P_\rad^{(n) \bot}|v|\big)\right\|^2
\\
 \le &  \left\|  \Gamma_q^{(n)}[f]\big|_{\C^{\B_n(T_q)}_\rad} \right\|^2  \cdot  \Big \|  P_\rad^{(n)}|v|\Big \|^2 +  \left\| \Gamma_q^{(n)}[f]\big|_{\C^{\B_n(T_q)} \ominus \C^{\B_n(T_q)}_\rad} \right\|^2 \cdot \Big \|  P_\rad^{(n) \bot}|v| \Big\|^2
\\
 =  & m^2  \Big \|  P_\rad^{(n)}|v|\Big \|^2 + M^2   \Big \|  P_\rad^{(n) \bot}|v| \Big\|^2
\\
  < & M^2  \Big \|  P_\rad^{(n)}|v|\Big \|^2 + M^2   \Big \|  P_\rad^{(n) \bot}|v| \Big\|^2   \,\, (\textit{strict inequality follows from  \eqref{assump-m-M}, \eqref{non-zero-proj}})
\\
 = & M^2 \qquad  \qquad (\textit{this step follows from  \eqref{dec-1}}).
\end{align*}
Thus we get a contradiction  and complete the whole proof of Theorem \ref{thm-con-ext} in case (A1).

\subsection{Proof of Theorem \ref{thm-con-ext}: Cases (A2) and (A3)}

By \eqref{norm-sup-2}, for any integer $n\ge 1$,  we always have  the one-sided inequality
$
\| T_n(f)\| \le  \| \Gamma_q^{(n)}[f]\|.
$
Therefore, it remains to prove that the reverse inequality
\begin{align}\label{rev-ineq}
\| \Gamma_q^{(n)}[f]\| \le \| T_n(f)\|
\end{align}
in case (A2) and case (A3).

We shall need the following classical results on the Carath\'eodory-Toeplitz extension problem  and  the Carath\'eodory-F\'ejer-Schur extension problem.

\begin{theorem}[{See, e.g., \cite[Thm 1.2 and Thm 1.3]{Foias-lifting} and \cite[Thm IV.24]{Tsuji}}]\label{thm-ext-pos-Toep}
Any finite non-negative definite Toeplitz matrix can be extended to  a non-negative definite  Toeplitz kernel on $\N$.  More precisely, fix an integer $n\ge 1$, if a finite Toeplitz matrix $T_n(f)$ is non-negative, then there exists a positive Radon measure on $\T$ such that
\[
T_n(f) = T_n(\mu),
\]
where
\[
T_n(\mu) = \Big[\int_\T e^{-i (k - l)\theta} d \mu(\theta)\Big]_{k, l \in \N}
\]
\end{theorem}

\begin{theorem}[{See, e.g., \cite[(3.2.3) in p.232]{Nikolski} and \cite[Thm 6.7]{Foias-lifting}  and \cite{CF}}]\label{thm-CF}
For any $a_0, a_1, \cdots, a_n \in \C$, we have
\[
\min\Big\{\| f\|_\infty\Big| \text{$f \in H^\infty(\T)$ and $\widehat{f}(k) = a_k$ for all $0\le k \le n$} \Big\} = \left\| T_n\Big( \sum_{k  =0}^n a_k e^{i k \theta} \Big)\right\|,
\]
where
\[
T_n\Big( \sum_{k  =0}^n a_k e^{i k \theta} \Big)  =
\left[
\begin{array}{cccccc}
a_0 & 0 & 0& \cdots & 0&0
\\
a_1 & a_0 & 0 & \cdots & 0 & 0
\\
a_2 & a_1 & a_0 & \cdots & 0 & 0
\\
\vdots & \vdots & \vdots & \vdots & \vdots & \vdots
\\
a_{n-1} & a_{n-2}& a_{n-3} & \cdots & a_0 & 0
\\
a_n & a_{n-1}& a_{n-2} & \cdots & a_1 & a_0
\end{array}
\right].
\]
\end{theorem}

{\flushleft \bf Case (A2).}
Fix an integer $n \ge 1$ and assume that $T_n(f)$ is Hermitian (this is equivalent to assume that $\Gamma_q^{(n)}[f]$ is Hermitian). By homogeneity, we may assume $\| T_n(f)\| = 1$. Then, by the Hermitian assumption on $T_n(f)$, both matrices $T_n(1 \pm f)$ are positive definite.  Therefore, by Theorem \ref{thm-ext-pos-Toep}, there exist two positive Radon measures $\mu_\pm$ on $\T$ such that for our fixed integer $n$, we have
\begin{align}\label{T-f-T-mu}
T_n(1 \pm f) = T_n(\mu_\pm).
\end{align}
Now by  Theorem \ref{thm-QW-I}, the infinite branching-Toeplitz kernels $\Gamma_q[\mu_\pm]$ are positive definite.  It follows that the restrictions $\Gamma_q^{(n)}[\mu_\pm]$ on the subset $\B_n(T_q) \times \B_n(T_q) \subset T_q \times T_q$ are  both  positive definite.  However the equalities \eqref{T-f-T-mu} for the integer $n$ imply the following equalities for the same $n$:
\[
\Gamma_q^{(n)}[1 \pm f] = \Gamma_q^{(n)}[\mu_\pm].
\]
Hence both matrices $\Gamma_q^{(n)}[1 \pm f]$ are positive definite. Note that we clearly have
\[
\Gamma_q^{(n)}[1 \pm f] = Id_{\B_n(T_q)} \pm \Gamma_q^{(n)}[f],
\]
where $Id_{\B_n(T)q)}$ is the identity matrix on $\C^{\B_n(T_q)}$. The desired  operator norm inequality $
\| \Gamma_q^{(n)}[f]\| \le 1$ then follows immediately.

\bigskip

{\flushleft \bf Case (A3).}
Fix an integer $n \ge 1$ and assume that $f \in H^\infty(\T)$. Note that we have
\[
T_n(f) = T_n\Big(\sum_{k = 0}^n \widehat{f}(k) e^{i k \theta}\Big).
\]
By Theorem \ref{thm-CF}, there exists $F\in H^\infty(\T)$ with $\| F\|_\infty = \| T_n(f)\|$ such that
\begin{align}\label{T-f-F}
T_n(f) = T_n(F).
\end{align}
Therefore, by Theorem \ref{thm-all-symbol}, we have $\| \Gamma_q[F]\| = \| F\|_\infty = \| T_n(f)\|$, which implies immediately  the inequality $\|\Gamma_q^{(n)}[F]\| \le \| \Gamma_q[F]\| = \| T_n(f)\|$.  But the equality \eqref{T-f-F} implies $\Gamma_q^{(n)}[f] = \Gamma_q^{(n)}[F]$, hence we obtain the desired inequality $\| \Gamma_q^{(n)}[f]\| \le \| T_n(f)\|$.

\subsection{Proof of Proposition \ref{prop-low-bdd}}  Fix $n\ge 1$ and a formal Fourier series $f$. Take any $g\in L^\infty(\T)$
 such that $\widehat{g}(k) = \widehat{f}(k)$ for all integers $k$ with $|k|\le n$. Then for any integer $q\ge 2$,  we have 
$\Gamma_q^{(n)}[f] = \Gamma_q^{(n)}[g]$.
Therefore,  by Theorem \ref{thm-all-symbol}, we have 
\[
\| \Gamma_q^{(n)}[f]\| = \| \Gamma_q^{(n)}[g]\| \le \| \Gamma_q[g]\|  = \| g\|_\infty. 
\] 
Thus by definition of $c_n(f)$, we have  $\| \Gamma_q^{(n)}[f]\| \le c_n(f)$ for any integer $q\ge 2$.  The proof is complete. 
\section{Appendix}
Here we give an alternative proof of the inequality \eqref{low-bdd} for a Hermitian kernel $\Gamma_q[f]$ by using Theorem \ref{thm-QW-I}.

 Assume that a formal Fourier series $f$ gives rise to a Hermitian kernel $\Gamma_q[f]$ and   $\lambda = \| \Gamma_q[f]\|<\infty$. Then $\lambda \cdot Id \pm \Gamma_q[f]$ are two positive definite kernels on $T_q$.   Since $\Gamma_q[1] = Id$, we obtain two positive definite kernels  $\Gamma_q[\lambda \pm f]$.
Then by applying Theorem \ref{thm-QW-I}, there exists two positive Radon measures $\mu_{\pm}$  on $\T  = \R /2 \pi \Z$ such that
\[
\lambda \delta_0(n) \pm \widehat{f}(n) = \int_\T e^{-i n \theta} d\mu_{\pm}(\theta), \quad \forall n \in \Z.
\]
Therefore, we have
\[
\pm \widehat{f}(n) =  \int_\T e^{- in \theta} d\Big[ \mu_{\pm} - \lambda m\Big](\theta), \quad \forall n \in \Z,
\]
where $dm$ is the normalized Haar measure on $\T$.  Hence we have
\[
\mu_{+} - \lambda m = - (\mu_{-} - \lambda m) \text{\, and thus \, }
\mu_{+} + \mu_{-} = 2\lambda m.
\]
It follows that both $\mu_{\pm}$ are absolutely continuous with respect to $m$. Denote $\mu_{\pm} = g_{\pm}  \cdot m$.
Then $g_{+}, g_{-} \ge 0$, $g_{+} + g_{-}  = 2\lambda$ and $f = g_{+}- \lambda = \lambda - g_{-}$.
Therefore, we have
\[
f = \frac{(g_{+}- \lambda) + (\lambda - g_{-})}{2} = \frac{g_{+} - g_{-}}{2}
\]
and thus we obtain the desired inequality
\[
\| f\|_\infty  =  \frac{| g_{+}- g_{-}|}{2} \le  \frac{ g_{+}+ g_{-}}{2} = \lambda = \| \Gamma_q[f]\|.
\]

\section*{Acknowledgements}
The research of Y. Qiu is supported by grants NSFC Y7116335K1,  NSFC 11801547 and NSFC 11688101 of National Natural Science Foundation of China. Z. Wang is supported by NSFC 11601296.

\end{document}